\documentclass[times]{amsart}

\usepackage{amsmath,amssymb}
\usepackage{amsthm, tikz}
\usepackage{hyperref}
\usepackage{cite}
\usepackage{graphicx}
\usepackage{color}
\usepackage{verbatim}

\newtheorem{thm}{Theorem}
\newtheorem*{theorem1*}{Theorem \ref{thm1}}
\newtheorem*{theorem2*}{Theorem \ref{thm2}}
\newtheorem{q}{Question}
\newtheorem{p}{Problem}
\newtheorem{lem}{Lemma}

\newtheorem{cor}{Corollary}
\newtheorem{obs}{Observation}
\newtheorem{prop}{Proposition}
\theoremstyle{definition}
\newtheorem{defn}{Definition}
\newtheorem{remark}{Remark}
\newtheorem{exam}{Example}

\DeclareMathOperator{\spn}{span}
\DeclareMathOperator{\sgn}{sgn}

\DeclareMathOperator{\ind}{in}
\DeclareMathOperator{\outd}{out}

\newcommand\R{\mathbb{R}}
\newcommand\C{\mathcal{C}}
\newcommand\Z{\mathbb{Z}}
\newcommand\Hy{\mathcal{H}}
\newcommand\Res{\mathcal{R}}
\newcommand\K{\mathcal{K}}
\newcommand\kg{\mathbf{K}}
\newcommand\kp{\kappa}
\newcommand\In{\mathcal{I}}
\newcommand\Th{\mathcal{T}}

\def\v{{\bf a}}

 \def\f_H{{\bf w}}
 
 \def\xx{{\bf x}}

\def\R{\mathbb{R}}

\def\Z{\mathbb{Z}}

 \def\F{\mathcal{F}}

\def\P{\mathcal{P}}

\newtheoremstyle{named}{}{}{\itshape}{}{\bfseries}{.}{.5em}{#1 \thmnote{#3}}
\theoremstyle{named}

\title{Root cones and the resonance arrangement}

\author{Samuel C. Gutekunst}
\author{Karola M\'esz\'aros}
\author[T. K. Petersen]{T. Kyle Petersen}
\thanks{Work of Gutekunst supported by the National Science Foundation Graduate Research Fellowship Program under Grant No. DGE-1650441. Work of M\'esz\'aros partially supported by a National Science Foundation Grant (DMS 1501059)  as well as by a von Neumann Fellowship at the IAS   funded by the Fund for Mathematics and the Friends of the Institute for Advanced Study. Work of Petersen supported by a Simons Foundation collaboration travel grant.  Any opinions, findings, and conclusions or recommendations expressed in this material are those of the authors and do not necessarily reflect the views of the National Science Foundation. 
}

\address{Operations Research and Information Engineering, Cornell University, Ithaca, NY 14853}
\address{Department of Mathematics, Cornell University, Ithaca, NY 14853 and School of Mathematics, Institute for Advanced Study, Princeton, NJ 08540}

\address{Department of Mathematical Sciences, DePaul University, Chicago, IL 60614}

\email{scg94@cornell.edu, karola@math.cornell.edu, tpeter21@depaul.edu }

\date{\today}

\begin{document}

\maketitle

\begin{abstract} 
We study the connection between triangulations of a type $A$ root polytope and the resonance arrangement, a hyperplane arrangement that shows up in a surprising number of contexts. Despite an  elementary definition for the resonance arrangement, the number of resonance chambers has only been   computed up to the $n=8$ dimensional case. We focus on data structures for labeling chambers, such as sign vectors and sets of alternating trees, with an aim at better understanding the structure of the resonance arrangement, and, in particular, enumerating its chambers. Along the way, we make connections with similar (and similarly difficult) enumeration questions. With the root polytope viewpoint, we  relate resonance chambers to the chambers of polynomiality of the Kostant partition function. With the hyperplane viewpoint, we clarify the  connections between resonance chambers and threshold functions. In particular, we show that  the base-2 logarithm of the number of resonance chambers is asymptotically $n^2$.
\end{abstract}

\section{Introduction}

This is a story of three counting problems: 
\begin{enumerate}
\item the number of chambers of polynomiality of the Kostant partition function,
\item the number of threshold functions, and 
\item the number of maximal unbalanced families.
\end{enumerate}

All three counting problems have resisted exact enumeration beyond small cases. We find in Sloane's On-Line Encyclopedia of Integer Sequences \cite{Sloane} that problem (1) has 6 entries (A119668), problem (2) has 10 entries (A000609), and problem (3) has 8 entries (A034997). The purpose of this article is to provide some links between these problems and to suggest some data structures that might prove useful for either exact or asymptotic enumeration.

\subsection{Kostant chambers}

Vector partition functions are fundamental in mathematics. A special vector partition function associated to the type $A_n$ root system is the \emph{Kostant partition function}, which was introduced by Bertram Kostant in 1958 in order to write down the multiplicity of a weight of an irreducible representation of a semisimple Lie algebra, also known as the \emph{Weyl character formula} or \emph{Kostant multiplicity formula} \cite{kostant2, kostant1}. Kostant partition functions are ubiquitous in mathematics, appearing not only in representation theory, but in algebraic combinatorics, toric geometry and approximation theory, among other areas.

The Kostant partition function is a piecewise polynomial function \cite{sturm} whose domains of polynomiality are maximal convex cones in the common refinement of all triangulations of the convex hull of the positive roots (see \cite{deloerasturm}), which we will refer to as \emph{Kostant chambers}. Let $\K_n$ denote this collection of cones, and let $K_n=|\K_n|$ denote the number of Kostant chambers. For example, Figure \ref{fig:A3Kostant} shows the seven chambers of $\K_3$. 

The inspiration for our work stems from an open problem posed by Kirillov \cite{ubiquity} and its investigation by de Loera and Sturmfels in \cite{deloerasturm}.
 
\begin{q} \label{q1} 
How many chambers of polynomiality does the Kostant partition function have?
\end{q}

It is an open problem to show that enumerating $K_n$ is \#P-hard \cite{triang}. The values of $K_n$ have been calculated up to $n=6$ by de Loera and Sturmfels \cite{deloerasturm}. See Table \ref{tab:data}. 

\begin{figure}
\[
\begin{tikzpicture}[scale=1.5]
\coordinate (a) at (0,4);
\coordinate (b) at (2,0);
\coordinate (c) at (-2,0);
\coordinate (d) at (1.55,2.36);
\coordinate (e) at (-1.55,2.36);
\coordinate (f) at (0,1.2);
\draw (a) to [in=90,out=-30]  (b);
\draw (b) to [in=-20,out=200]  (c);
\draw (c) to [in=210,out=90]  (a);
\draw (c) to [in=220,out=30] (d);
\draw (b) to [in=-40,out=150] (e);
\draw (a) to (f);
\draw (d) to [in=10,out=170] (e);
\draw (a) node[fill=black,circle, inner sep=2] {} node[above] {$e_2-e_3$};
\draw (b) node[fill=black,circle, inner sep=2] {} node[right] {$e_3-e_4$};
\draw (c) node[fill=black,circle, inner sep=2] {} node[left] {$e_1-e_2$};
\draw (d) node[fill=black,circle, inner sep=2] {} node[right] {$e_2-e_4$};
\draw (e) node[fill=black,circle, inner sep=2] {} node[left] {$e_1-e_3$};
\draw (f) node[fill=black,circle, inner sep=2] {} node[below,yshift=-5pt] {$e_1-e_4$};
\end{tikzpicture}
\]
\caption{A two-dimensional slice of the set of cones in $\K_3$, i.e., the 7 domains of polynomiality for the Kostant partition function.}\label{fig:A3Kostant}
\end{figure}
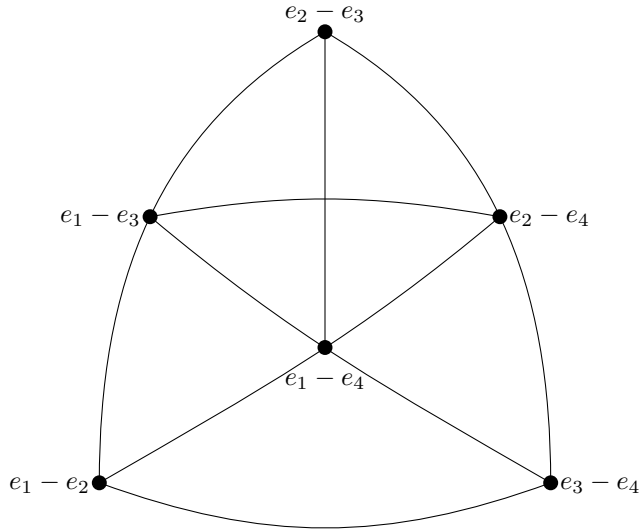

\subsection{Threshold functions}

The study of \emph{linear threshold functions} has a long history of applications in a variety disciplines, including Economics, Psychology, and Computer Science \cite{MurogaBook}. These are boolean functions $f:\{-1,1\}^n \to \{-1,1\}$ of the form $f(\mathbf{x}) = \sgn( t + \mathbf{a}\cdot \mathbf{x})$ for some \emph{threshold} $t$ and some vector $\mathbf{a}$ known as the \emph{weight vector}. 

It is well-known that threshold functions correspond to their weight vectors $\mathbf{a}$ only up to the half-spaces determined by negative/nonnegative dot products with $\pm 1$ vectors (see, e.g., \cite{Zuev92}). That is, threshold functions are in bijection with the chambers in the hyperplane arrangement whose normal vectors are all $\pm 1$ vectors, representing vertices of an $(n+1)$-cube. Let $\Th_{n+1}$ denote this arrangement of hyperplanes, which we call the \emph{threshold arrangement}, and let $T_n$ denote the number of chambers in this arrangement, i.e., the number of threshold functions on $n$ variables. See Figure \ref{fig:thresharr3} for the rank 3 arrangement. More details will come in Section \ref{sec:threshold}. According to \cite[A000609]{Sloane}, the largest known exact value for $T_n$ is $T_9=144\,13053\,14531\,21108$, computed in 2006 by work of Gruzling \cite{Gruzling06}. See Table \ref{tab:data}. While exact values are in short supply, some asymptotic estimates for $T_n$ have been made. The best estimate we know of comes from work of Zuev \cite {Zuev92}, which shows that $\log_2 T_n \sim n^2$.

\begin{figure}
\includegraphics[scale=.5]{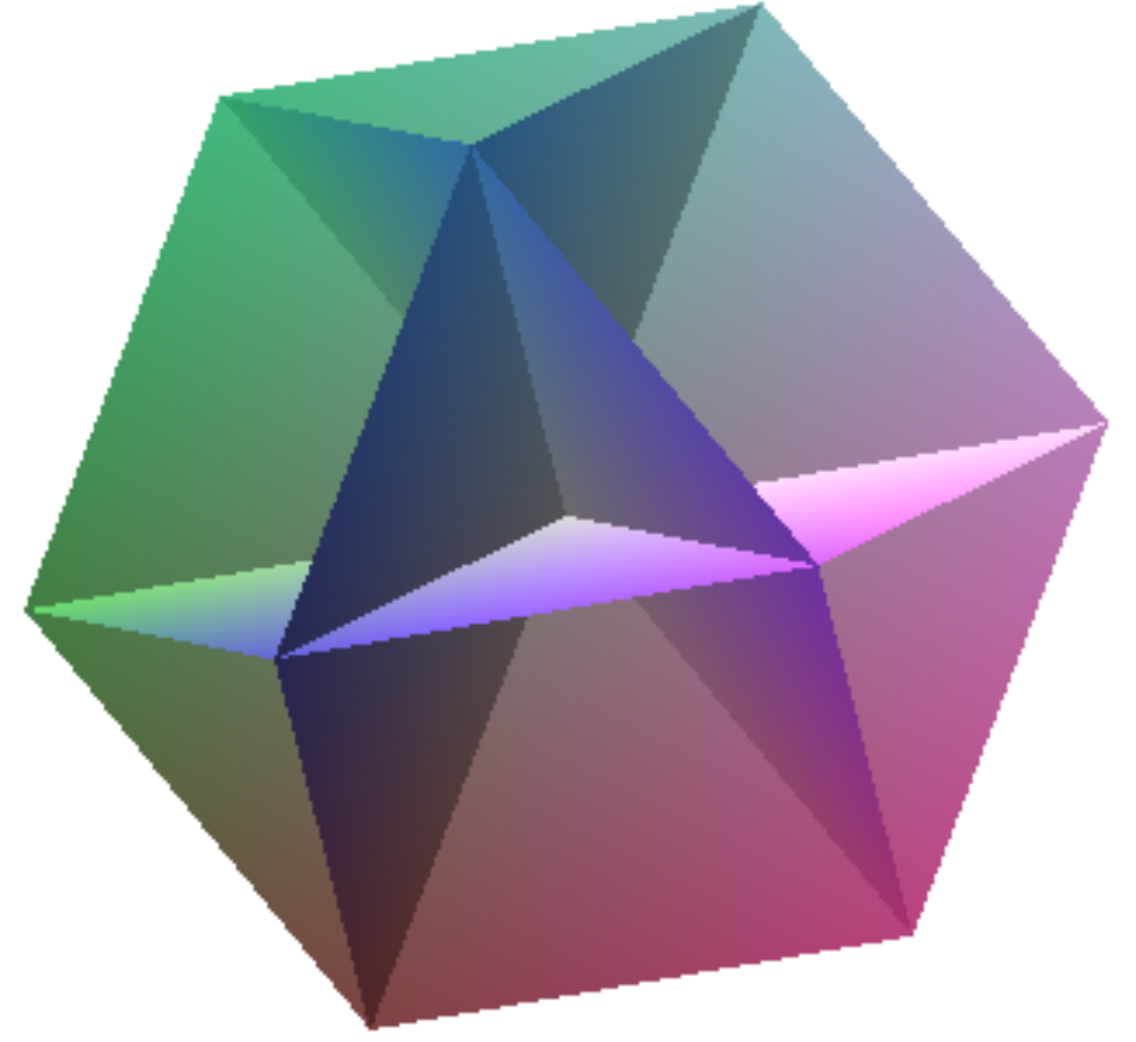} 
\caption{A view of the threshold arrangement $\Th_3$ of rank 3. In the image we can see seven chambers above $V_{\emptyset}$ (the other seven are antipodal to these), thus there are $T_2=14$ threshold functions on two variables.}\label{fig:thresharr3}
\end{figure}

\subsection{Maximal unbalanced families}

While perhaps less well-known, \emph{maximal unbalanced families} have appeared in a surprising number of guises. A family of subsets, which we think of as a collection of vertices of an $n$-cube $\{0,1\}^n$, is \emph{balanced} if the convex hull of the vertices intersects the diagonal of the $n$-cube. A family is \emph{unbalanced} otherwise. Shapley and others studied balanced families in the context of game theory \cite{Shapley67}. Balanced families are closed under taking unions, and hence some of Shapley's results are phrased in terms of \emph{minimal balanced families}. Dually, the collection of unbalanced families is closed under taking intersections, which inspired  the investigation of \emph{maximal unbalanced families}. In the work of Billera, Tatch Moore, Dufort Moraites, Wang, and Williams \cite{Bil12}, it is recognized that maximal unbalanced families are in bijection with chambers of a hyperplane arrangement which we refer to as the \emph{resonance arrangement}, following \cite{Billera18, Cavalieri11}.

The resonance arrangement appears in several places: For example, Kamiya, Takemura, and Terao studied this arrangement with relation to ``ranking patterns of unfolding models'' which have found applications in Psychometrics, Marketing, and Voting Theory \cite{Kam11, Kam12}.\footnote{Kamiya, Takemura, and Terao call the resonance arrangement the \emph{all-subsets arrangement}, and that name is also used by Billera, Tatch Moore, Dufort Moraites, Wang, and Williams. We adopt the nomenclature of Cavalieri et al.\, which is also followed in later work on Hurwitz numbers and is used in recent work of Billera, Billey, and Tewari \cite{Billera18}.  In Liu, Norledge, and Ocneanu \cite{Liu19},  the resonance arrangement is also called the adjoint braid arrangement.} In the case of ranking patterns of codimension one, they find the patterns in bijection with maximal unbalanced families. In Physics, Evans encountered and enumerated ``generalized retarded functions'' when studying the analytic continuations of thermal Green functions \cite{Evans91,Evans94} of low rank, and it happens that these functions are in bijection with maximal unbalanced families as well. Recent mathematical work has also connected to unbalanced families and the resonance arrangement: Cavalieri et al.\  show that the chambers of the resonance arrangement correspond to domains of polynomiality for double Hurwitz numbers \cite[Theorem 1.3]{Cavalieri11}; Bj\"orner used combinatorial topology to make a connection between maximal unbalanced families and a conjecture from extremal combinatorics \cite{Bjorner15}; and Lewis, McCammond, Petersen, and Schwer found that the local distribution of reflection length in the affine symmetric group is generic in chambers of the resonance arrangement \cite[Proposition 3.2(ii)]{Lewis17}.  See also Early \cite{Early17, Early18}.

While it can be defined in several equivalent ways, we will see that the resonance arrangement, denoted $\Res_n$, is isomorphic to the intersection of the threshold arrangement with the hyperplane $\{ \mathbf{x} \in \R^{n+1} : \sum x_i = 0\}$. See Figure \ref{fig:rank3res}. We let $R_n$ denote the number of chambers of the resonance arrangement, i.e., the number of maximal unbalanced families on $\{0,1\}^{n+1}$. The largest known value of $R_n$ according to OEIS is $R_8 = 41\,91727\,56930$ contributed in 2011 by Evans \cite[A034997]{Sloane}. From general properties of hyperplane arrangements it follows that the number of resonance chambers has roughly the same asymptotic behavior as the number of threshold functions, so $\log_2 R_n \sim n^2$ as well. We make this and other claims precise in the next subsection.

\begin{figure}
\includegraphics[scale=.5]{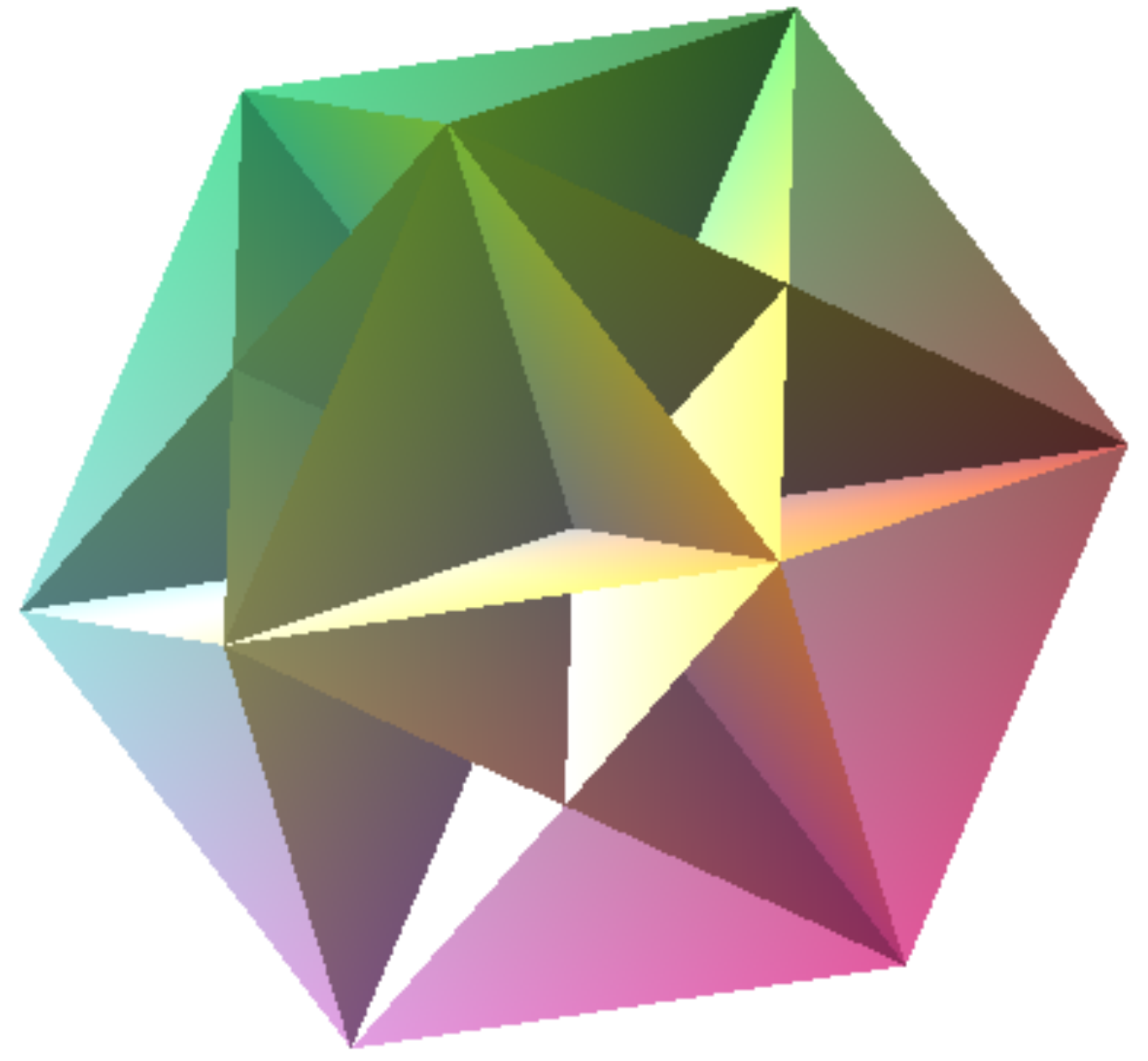}
\caption{The rank three resonance arrangement $\Res_3$ projected onto the $V_{\emptyset} = \{ (w,x,y,z) : w+x+y+z=0\}$ hyperplane. There are 16 chambers visible, and another 16 antipodal to these, so $R_3=32$.}\label{fig:rank3res}
\end{figure}

\subsection{Results and questions}

We now state some results relating $K_n, T_n$, and $R_n$. In Table \ref{tab:data} we compare the sequences in various ways.

\begin{table}
\begin{center}
\begin{tabular}{|c|c|c|c|c|c|c|} \hline 
$n$ &  $\left\lfloor \frac{(n+1)}{2^{n+1}}T_n\right\rfloor$ & $R_n$ & $K_{n+1}$  & $\frac{1}{2}T_n$ & $\frac{R_{n+1}}{(n+2)}$ & $2^{n^2}$  \\ 
\hline \hline
 1 &  2 & 2 & 2 & 2 & 2 & 2 \\
 2 &  5 & 6 & 7 & 7 & 8 & 16 \\
 3 &  26 & 32 & 48 & 52 & 74 & 512 \\
 4 &  294 & 370 & 820 & 941 & 1882 & 65536 \\
 5 &  8866 & 11292 & 44288 & 47286 & 152292 & 33554432 \\
 6 &  821851 & 1066044 & ?  & 7514067 & 43415794 & 68719476736 \\
 7 &  {\small 261814714} & {\small 347326352} & ? & {\small 4189035432} & {\small 46574750770} & {\small 562949953421312}  \\
 8 & {\tiny 308698937454} & {\tiny 419172756930} & ? & {\tiny 8780769776473} & ? & {\tiny 18446744073709551616} \\
 \hline
\end{tabular}
\end{center}
\caption{Comparisons between $K_n$ (the number of Kostant chambers), $T_n$ (the number of threshold functions), and $R_n$ (the number of maximal unbalanced families).}\label{tab:data}
\end{table}

\begin{remark}[Indexing of sequences]
The indexing of $K_n$ matches the dimension of the positive root cone, i.e., the rank of the root system $A_n$. This is in agreement with other work, such as \cite{deloerasturm}. We caution however that we will use collections of trees on $[n+1]$ to label Kostant chambers. 

The number $T_n$ is the number of linear threshold functions on $n$ variables, but the threshold arrangement $\Th_{n+1}$ has rank $n+1$. For example, there are four one-variable threshold functions, $T_1=4$, corresponding to four cones in a plane, and $T_2=14$ counts the two-variable threshold functions, corresponding to fourteen chambers in the arrangement of planes in Figure \ref{fig:thresharr3}. We choose to align our index with the number of variables in the corresponding threshold function, since that convention is well-established in the literature.

The indexing for $R_n$ matches the rank of the resonance arrangement, with $R_1 = 2$, $R_2=6$, $R_3=32$, and so on. This indexing is chosen for our convenience; it differs with some conventions used for counting its chambers, e.g., in \cite{Bil12}, they use $E_2 =2$, $E_3=6$, and so on, $E_n = R_{n-1}$. In that work, the focus was on maximal unbalanced families, and the  subscript on $E_n$ corresponds to the cardinality of the set from which the family of subsets is drawn. That is, $E_n=R_{n-1}$ is the number of maximal unbalanced families formed from an $n$-element set.
\end{remark}

Prior work on estimating $R_n$ and $T_n$ shows that they are both on the order of $2^{n^2}$. In particular, Zuev \cite{Zuev92} shows that for $n\geq 2,$
\[
 2^{n^2(1-10/\ln(n))} < T_n < 2^{n^2},
\]
which implies that 
\[
 \log_2 T_n \sim n^2.
\]
Similarly, Billera et al.\  \cite{Bil12} show that for $n\geq 2$
\[
 2^{\frac{n(n-1)}{2}} < R_n < 2^{n^2},
\] 
implying that $\log_2 R_n \sim cn^2$ for some $1/2\leq c \leq 1$.

One of our results, first observed  by Billera \cite{LouSlides}, is improved bounds on $R_n$, as given here\footnote{We note that Theorem 1.4 of Deza, Pournin, and Rakotonarivo \cite{Dez19} gives a tighter upper bound \\ $R_n\leq 2(n+4)2^{n^2-3n+2},$ also implying that $\log_2 R_n < n^2.$ }.

\begin{thm}\label{thm:Rn}
For any $n\geq 2$,
\begin{equation}\label{eq:ineq}
 \frac{(n+1)}{2^{n+1}}T_n < R_n < \frac{1}{2}T_n,
\end{equation}
and therefore
\[
  n^2 - 10n^2/\ln(n) -n + \log_2(n+1) -1 < \log_2 R_n < n^2-1,
\]
so $\log_2 R_n \sim n^2$.
\end{thm}

We also record the following natural inequality relating the number of Kostant chambers to the number of resonance chambers:

\begin{obs} \label{obs:Kn} 
For any $n\geq 2$,
\begin{equation}\label{eq:Kn}
 K_n \leq \frac{R_n}{n+1},
\end{equation}
and in particular,
\[
  \log_2 K_n \leq n^2-1-\log_2(n+1).
\]
\end{obs}

By the numerical evidence in the table we also propose the following problem:

\begin{p}\label{conj:bounds}
Is it true that for any $n\geq 3$,
\begin{equation}\label{eq:conjKn}
R_n < K_{n+1} < \frac{1}{2}T_n,
\end{equation}
and in particular
$\log_2 K_{n+1} \sim n^2$?
\end{p}

As Table \ref{tab:data} shows we have only five data points suggesting a positive answer to Problem \ref{conj:bounds}. We will also  provide  combinatorial models for chambers that makes links between the three sequences seem more plausible.

The method of proof for Theorem \ref{thm:Rn} is to carefully investigate the structure of the hyperplane arrangements $\Th_n$ and $\Res_n$ using the standard notion of a \emph{sign vector} for encoding chambers. As we will see, Observation \ref{obs:Kn} follows readily from chamber combinatorics.

We also put the combinatorics of root polytopes  pioneered by Postnikov \cite{Post09} to use. In particular we consider chambers of $\K_n$ and $\Res_n$ to be labeled by certain sets of \emph{alternating trees}.  We find it useful to define a graph $\Gamma_n$ whose vertices are all alternating trees on $[n]=\{1,2,\ldots,n\}$. We determine the adjacency of two trees via the notion of \emph{sign compatibility}---a purely graph-theoretic condition that implies the corresponding root simplices have full-dimensional intersection. The graph $\Gamma_n$, which we call the \emph{compatibility graph} also has a subgraph $\Gamma_n^+$, with the same adjacency relation, whose vertices are positive alternating trees, which label positive root simplices.  We establish the following result.

\begin{thm}\label{thm:cliques}
The chambers of the resonance arrangement $\Res_n$ can be labeled by cliques in the compatibility graph $\Gamma_n$, and the Kostant chambers $\K_n$ can be labeled by cliques in $\Gamma_n^+$. Moreover, chambers of the resonance arrangement are in bijection with a subset of the maximal cliques in $\Gamma_n$. 
\end{thm}

In later sections we propose several problems and questions about characterizing precisely which cliques correspond to the various types of chambers.

\subsection{Organization of the paper}

The paper is divided into four main sections. Section \ref{sec:data} introduces the key data structures that we use for labeling chambers, namely \emph{sign vectors} and \emph{alternating trees}.  
In Section \ref{sec:resonance} we introduce the key definitions for the study of the resonance arrangement and show how to label chambers with both sign vectors and with collections of alternating trees. Section \ref{sec:graph} in particular introduces the graph discussed in Theorem \ref{thm:cliques}. In Section \ref{sec:Kostant} we turn our attention onto the problem of counting Kostant chambers, and we observe that Kostant chambers are unions of resonance chambers in the positive root cone. In Section \ref{sec:threshold}, we turn our focus to the links between the resonance arrangement and the threshold arrangement, culminating in the  proof of Theorem \ref{thm:Rn}.

\section{Data structures for root polytopes and hyperplane arrangements}\label{sec:data}

In this section we establish some basic notions that will be used throughout the paper.

\subsection{Sign vectors}

Let $V$ be a finite-dimensional real vector space. A \emph{hyperplane} $H$ is a subspace of codimension 1. A \emph{hyperplane arrangement} $\Hy = \{ H_i \}_{i\in I}$ is a collection of hyperplanes in $V$ indexed by the set $I$. To each hyperplane we associate a nonzero normal vector $v_i$ so that $H_i = \{ \lambda \in V : \langle \lambda, v_i \rangle = 0\}$. Similarly, let the positive and negative half-spaces of $H_i$ be defined by $H_i^+= \{ \lambda \in V : \langle \lambda, v_i \rangle  > 0\}$ and $H_i^-= \{ \lambda \in V : \langle \lambda, v_i \rangle  < 0\}$.  The \emph{rank} of a hyperplane arrangement is $\dim (\spn\{v_i\}_{i\in I})$. 
  
Following \cite{Abramenko08}, it is known that the hyperplane arrangement $\Hy$ partitions $V$ into a collection of disjoint convex cones called \emph{faces} given by intersections of hyperplanes and their half-spaces. A face $F$ is uniquely determined by its  \emph{sign vector}: \[ \sigma(F) = (\sigma_i(F))_{i \in I},\] where $\sigma_i(F)= +, -,$ or 0 to indicate whether, for points $\lambda \in F$, we have $\langle \lambda, v_i\rangle > 0$, $<0$, or $=0$, respectively. Said another way, \[ F = \bigcap_{i \in I} H_i^{\sigma_i(F)},\] where $H^0_i = H_i$.

There is a natural partial order on faces, given by $F\leq G \Leftrightarrow \overline{F} \subseteq \overline{G}$; that is, if the closure of $F$ is contained in the closure of $G$. In terms of sign sequences, this can be stated as: $F \leq G$ if and only if for each $i \in I$ either $\sigma_i(F) = 0$ or $\sigma_i(F) = \sigma_i(G)$. 

This partial order is ranked by dimension, and maximal faces are called \emph{chambers}. They are characterized by the fact that $\sigma_i(C)\neq 0$ for all $i \in I$. This also means that chambers are the maximal connected components in $V-\Hy$. Codimension one faces are called \emph{walls}. We can see that a face $F$ is a wall whenever $\sigma_i(F)=0$ for precisely one entry in $\sigma(F)$.

For example, in Figure \ref{fig:smallarrangement}, we see a line arrangement with three normal vectors giving rise to six one-dimensional walls and six two-dimensional chambers.

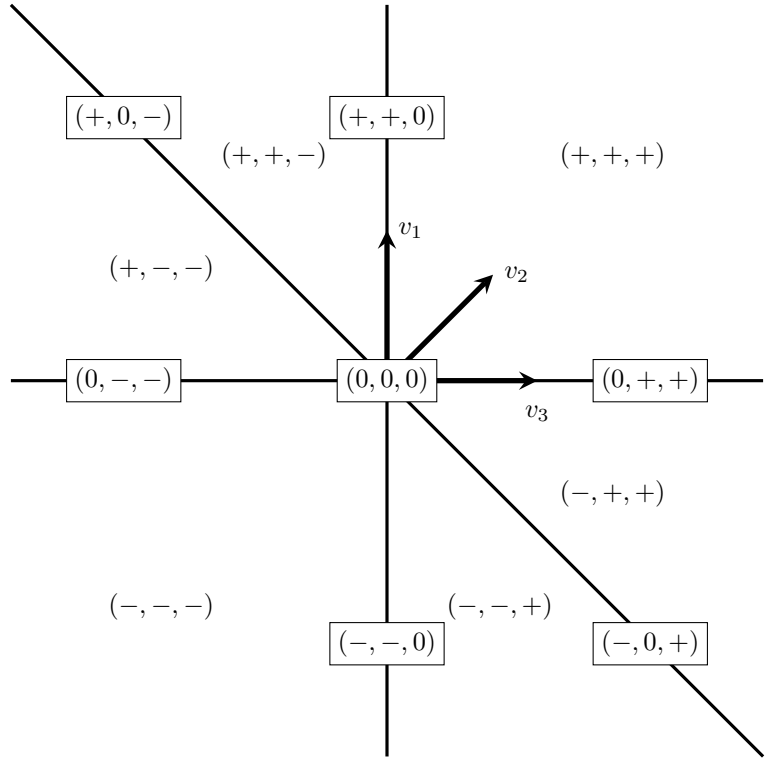
\begin{figure}
\[
\begin{tikzpicture}[>=stealth]
\draw[very thick] (-5,5) -- node[draw=black,thin,fill=white,pos=.15,rectangle,inner sep=3pt] {
$(+,0,-)$
} 
node[draw=black,thin,fill=white,pos=.85,rectangle,inner sep=3pt] {
$(-,0,+)$
} (5,-5);
\draw[very thick] (-5,0) -- node[draw=black,thin,fill=white,pos=.15,rectangle,inner sep=3pt] {
$(0,-,-)$
} 
node[sloped,draw=black,thin,fill=white,pos=.85,rectangle,inner sep=3pt] 
{
$(0,+,+)$
} (5,0);
\draw[very thick] (0,-5) -- node[draw=black,thin,fill=white,pos=.15,rectangle,inner sep=3pt] 
{
$(-,-,0)$
} 
node[draw=black,thin,fill=white,pos=.85,rectangle,inner sep=3pt] 
{
$(+,+,0)$
} 
(0,5);
\draw (3,3) node {
$(+,+,+)$
};
\draw (3,-1.5) node {
$(-,+,+)$
};
\draw (-1.5,3) node {
$(+,+,-)$
};
\draw (1.5,-3) node {
$(-,-,+)$
};
\draw (-3,1.5) node {
$(+,-,-)$
};
\draw (-3,-3) node {
$(-,-,-)$
};
\draw[->,line width =2] (0,0) -- (0,2) node[right] {$v_1$};
\draw[->,line width =2] (0,0) -- (1.41,1.41) node[right] {$v_2$};
\draw[->,line width =2] (0,0) -- (2,0) node[below,yshift=-5pt] {$v_3$};
\draw (0,0) node[inner sep=3,fill=white,draw=black] {
$(0,0,0)$
};
\end{tikzpicture}
\]
\caption{A line arrangement, i.e., a rank two hyperplane arrangement, with faces labeled by sign vectors $(\sigma_1,\sigma_2,\sigma_3)$ corresponding to normal vectors $v_1, v_2, v_3$.}\label{fig:smallarrangement}
\end{figure}

Let $C(\Hy)$ denote the number of chambers in $V-\Hy$. Let $W(H_i)$ denote the number of walls in hyperplane $H_i$, and let $W(\Hy)$ denote the total number of walls in the arrangement. Since each wall  appears in only one hyperplane, $W(\Hy)=\sum_{i\in I} W(H_i)$. Here is an easy observation that holds for any finite hyperplane arrangement.

\begin{obs}\label{obs:hyp}
We have the following facts in any finite hyperplane arrangement of rank $n$:
\begin{enumerate}
\item $2W(H_i) \leq C(\Hy)$, for any hyperplane $H_i$, and 
\item $nC(\Hy) \leq 2W(\Hy)$. 
\end{enumerate}
\end{obs}

\begin{proof}
Consider a wall $F$ with $\sigma_i(F) =0$ and all other sign vector entries nonzero. This face lies on the boundary of two chambers, each obtained by keeping the nonzero entries fixed and choosing $\sigma_i$ to be $+$ or $-$. Thus there are at least two distinct chambers for each wall of $H_i$. This proves the first observation.

For the second observation, we notice that in a rank $n$ arrangement, every chamber must have at least $n$ walls on its boundary, while each wall is on the boundary of precisely two chambers.
\end{proof}

In Section \ref{sec:threshold} we will exploit Observation \ref{obs:hyp} to prove Theorem \ref{thm:Rn}.

\subsection{Alternating trees}

Here we discuss a combinatorial data structure arising in Postnikov's work on root polytopes, which we will use in labeling both Kostant chambers and resonance chambers. Recall that the type $A_{n-1}$ root system is the set of vectors $\Phi=\{ e_i-e_j : 1\leq i,j \leq n, i \neq j\}$, with positive roots $\Phi^+ = \{ e_i-e_j : 1\leq i <j \leq n\}$. The linear span of the roots will be denoted by $V_{\emptyset} = \{ \mathbf{x} \in \R^n : \sum x_i =0\}$, which is a hyperplane of $\R^n$ that will be of interest to us later.

\begin{defn}[Root polytope]\label{def:rootpoly}
Given a directed graph $G$ on the vertex set $[n]$, with arc set $E(G)$, we associate to it the \textbf{root polytope} 
\begin{equation}\label{eq:rootpoly}
\mathcal{P}(G)=\text{conv}\{0, e_i-e_j: (i,j) \in E(G)\}. 
\end{equation}
\end{defn}

Two special cases of \eqref{eq:rootpoly} are as follows. If we take $G$ to be the complete graph $\kg_n$ (we use boldface to distinguish from the number of Kostant chambers $K_n$), then the polytope $\P(\kg_n)$ is the convex hull of all roots, which we refer to as the \emph{full root polytope}. If we let $\kg_n^{+}$ denote the complete graph on $[n]$ with edges only directed from smaller vertices to larger, then $\P(\kg_n^+)$ is the convex hull of the positive roots, which we call the \emph{positive root polytope}. Note that since roots live in the hyperplane $V_{\emptyset}$, the polytopes $\P(\kg_n)$ and $\P(\kg_n^+)$ are $(n-1)$-dimensional. We see these polytopes for $n=3$ in Figure \ref{fig:rootpoly2d}. 

\begin{lem} \label{lem:sim} (cf. \cite[Lemma 12.5]{Post09}) 
Given a directed graph $G$, the root polytope $\P(G)$ is a simplex with the origin as one of the vertices if and only if $G$ is acyclic. When $G$ is acyclic, the dimension of $\P(G)$ is the number of edges of $G$.
\end{lem}

Given an acyclic graph $F$, let $\Delta_F:=\P(F)$ to emphasize that $\P(F)$ is a simplex. (We remark that this notation differs from Postnikov, who uses $\tilde{\Delta}_F$ for our $\Delta_F$.) As maximal acyclic graphs, trees will be of particular interest.

\begin{figure}
\[
 \begin{array}{c c}
  \begin{tikzpicture}[cm={1,0,.5,.8660254,(0,0)}, >=stealth,baseline=2.6cm, scale=3]
       \filldraw[fill=black!20!white] (0,1) node[circle,fill=black,inner sep=2] {}--(1,0) node[circle,fill=black,inner sep=2] {}--(1,-1) node[circle,fill=black,inner sep=2] {}--(0,-1) node[circle,fill=black,inner sep=2] {} --(-1,0) node[circle,fill=black,inner sep=2] {}--(-1,1) node[circle,fill=black,inner sep=2] {}--(0,1)--cycle;
       \draw (0,1) node[above] {$e_1-e_3$};
       \draw (1,0) node[right] {$e_2-e_3$};
       \draw (-1,1) node[above] {$e_1-e_2$};
       \draw (1,-1) node[below] {$e_2-e_1$};
       \draw (0,-1) node[below] {$e_3-e_1$};
       \draw (-1,0) node[left] {$e_3-e_2$};
       \draw (0,1)--(0,-1);
       \draw (1,0)--(-1,0);
       \draw (-1,1)--(1,-1);
       \draw (0,0) node[circle,fill=black, inner sep=2] {} node[below left, xshift=-5pt] {$0$};
\draw (.35,.35) node[scale=.75] {
       \begin{tikzpicture}[baseline=0,>=stealth]
\foreach \x in {1,2,3}{
\coordinate (\x) at (\x,0);
}
 \draw (1) node[below,fill=white,draw=black,circle,inner sep =2] {};
 \draw (2) node[below,fill=white,draw=black,circle,inner sep =2] {};
 \draw (3) node[below,fill=black,draw=black,circle,inner sep =2] {};
\draw[->] (1) to [in=90,out=90]  (3);
\draw[->] (2) to [in=90,out=90]  (3);
\end{tikzpicture}
       };
\draw (-.4,.8) node[scale=.75] {
       \begin{tikzpicture}[baseline=0,>=stealth]
\foreach \x in {1,2,3}{
\coordinate (\x) at (\x,0);
}
 \draw (1) node[below,fill=white,draw=black,circle,inner sep =2] {};
 \draw (2) node[below,fill=black,draw=black,circle,inner sep =2] {};
 \draw (3) node[below,fill=black,draw=black,circle,inner sep =2] {};
\draw[->] (1) to [in=90,out=90]  (2);
\draw[->] (1) to [in=90,out=90]  (3);
\end{tikzpicture}
       };
\draw (-.4,-.2) node[scale=.75] {
       \begin{tikzpicture}[baseline=0,>=stealth]
\foreach \x in {1,2,3}{
\coordinate (\x) at (\x,0);
}
 \draw (1) node[below,fill=black,draw=black,circle,inner sep =2] {};
 \draw (2) node[below,fill=black,draw=black,circle,inner sep =2] {};
 \draw (3) node[below,fill=white,draw=black,circle,inner sep =2] {};
\draw[->] (3) to [in=90,out=90]  (1);
\draw[->] (3) to [in=90,out=90]  (2);
\end{tikzpicture}
       };
       \draw (-.65,.3) node[scale=.75] {
       \begin{tikzpicture}[baseline=0,>=stealth]
\foreach \x in {1,2,3}{
\coordinate (\x) at (\x,0);
}
 \draw (1) node[below,fill=white,draw=black,circle,inner sep =2] {};
 \draw (2) node[below,fill=black,draw=black,circle,inner sep =2] {};
 \draw (3) node[below,fill=white,draw=black,circle,inner sep =2] {};
\draw[->] (3) to [in=90,out=90]  (2);
\draw[->] (1) to [in=90,out=90]  (2);
\end{tikzpicture}
       };
       \draw (.65,-.25) node[scale=.75] {
       \begin{tikzpicture}[baseline=0,>=stealth]
\foreach \x in {1,2,3}{
\coordinate (\x) at (\x,0);
}
 \draw (1) node[below,fill=black,draw=black,circle,inner sep =2] {};
 \draw (2) node[below,fill=white,draw=black,circle,inner sep =2] {};
 \draw (3) node[below,fill=black,draw=black,circle,inner sep =2] {};
\draw[->] (2) to [in=90,out=90]  (3);
\draw[->] (2) to [in=90,out=90]  (1);
\end{tikzpicture}
       };
       \draw (.35,-.75) node[scale=.75] {
       \begin{tikzpicture}[baseline=0,>=stealth]
\foreach \x in {1,2,3}{
\coordinate (\x) at (\x,0);
}
 \draw (1) node[below,fill=black,draw=black,circle,inner sep =2] {};
 \draw (2) node[below,fill=white,draw=black,circle,inner sep =2] {};
 \draw (3) node[below,fill=white,draw=black,circle,inner sep =2] {};
\draw[->] (3) to [in=90,out=90]  (1);
\draw[->] (2) to [in=90,out=90]  (1);
\end{tikzpicture}
       };
\end{tikzpicture}
 &
\begin{tikzpicture}[cm={1,0,.5,.8660254,(0,0)}, >=stealth, baseline=4cm, scale=3]
       \filldraw[fill=black!20!white] (0,1) node[circle,fill=black,inner sep=2] {}--(1,0) node[circle,fill=black,inner sep=2] {}--(0,0) node[circle,fill=black,inner sep=2] {}--(-1,1) node[circle,fill=black,inner sep=2] {} --(0,1)--cycle;
       \draw (0,1) node[above] {$e_1-e_3$};
       \draw (1,0) node[below] {$e_2-e_3$};
       \draw (-1,1) node[above] {$e_1-e_2$};
       \draw (0,0)--(0,1);
       \draw (0,0) node[below left] {$0$};
       \draw (.35,.35) node[scale=.75] {
       \begin{tikzpicture}[baseline=0,>=stealth]
\foreach \x in {1,2,3}{
\coordinate (\x) at (\x,0);
}
 \draw (1) node[below,fill=white,draw=black,circle,inner sep =2] {};
 \draw (2) node[below,fill=white,draw=black,circle,inner sep =2] {};
 \draw (3) node[below,fill=black,draw=black,circle,inner sep =2] {};
\draw[->] (1) to [in=90,out=90]  (3);
\draw[->] (2) to [in=90,out=90]  (3);
\end{tikzpicture}
       };
\draw (-.4,.8) node[scale=.75] {
       \begin{tikzpicture}[baseline=0,>=stealth]
\foreach \x in {1,2,3}{
\coordinate (\x) at (\x,0);
}
 \draw (1) node[below,fill=white,draw=black,circle,inner sep =2] {};
 \draw (2) node[below,fill=black,draw=black,circle,inner sep =2] {};
 \draw (3) node[below,fill=black,draw=black,circle,inner sep =2] {};
\draw[->] (1) to [in=90,out=90]  (2);
\draw[->] (1) to [in=90,out=90]  (3);
\end{tikzpicture}
       };
\end{tikzpicture} \\
\P(\kg_3) & \P(\kg_3^+)
 \end{array}
\]
\caption{Central triangulations of the full root polytope $\P(\kg_3)$ and the positive root polytope $\P(\kg_3^+)$, drawn in the plane $x_1 + x_2 + x_3 = 0$.}\label{fig:rootpoly2d}
\end{figure}
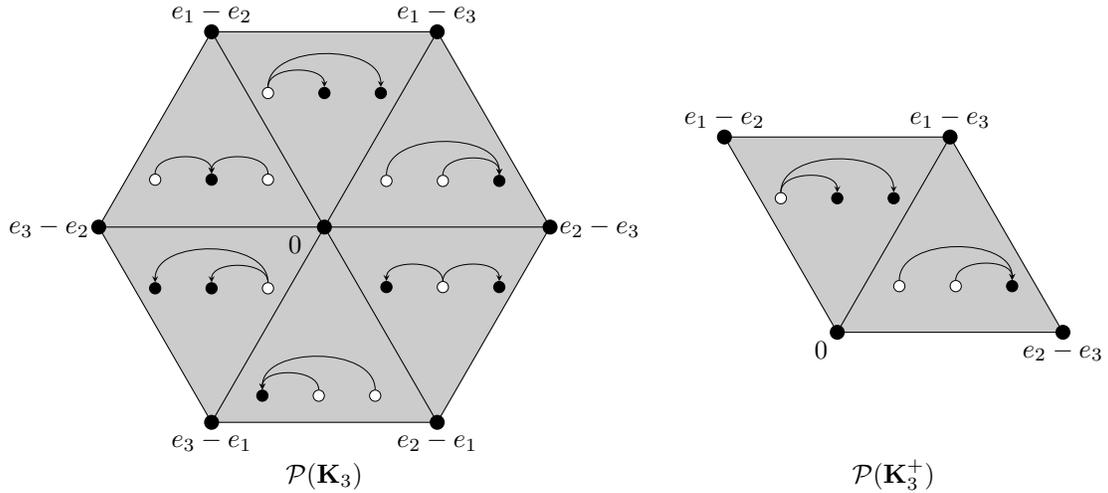

\begin{defn}[Alternating graph] \label{def:alt} A directed graph $G$ is {\bf alternating} if each vertex is either a source (all outgoing arcs) or a sink (all incoming arcs). A directed graph $G$ on $[n]$ is \textbf{positive alternating} if it is alternating and all arcs are of the form $(i,j)$ for some $i<j$.
\end{defn}

For example, in Figure \ref{fig:alternatingtrees} we see a tree $T$ that is alternating but not positive alternating and another tree $T'$ that is positive alternating. In examples such as these we label the sources with white nodes and the sinks with black nodes for easy identification.

\begin{figure}
\[
\begin{array}{c c c}
\begin{tikzpicture}[baseline=0,>=stealth,scale=.75]
\foreach \x in {1,...,7}{
\coordinate (\x) at (\x,0);
}
\foreach \x in {1,4,6,7}{
 \draw (\x) node[below,fill=white,draw=black,circle,inner sep =2] {} node[below,yshift=-5pt] {$\x$};
}
\foreach \x in {2,3,5}{
 \draw (\x) node[below,fill=black,draw=black,circle,inner sep =2] {} node[below,yshift=-5pt] {$\x$};
}
\draw[->] (1) to [in=90,out=90]  (2);
\draw[->] (1) to [in=90,out=90]  (3);
\draw[->] (4) to [in=90,out=90]  (3);
\draw[->] (4) to [in=90,out=90]  (5);
\draw[->] (6) to [in=90,out=90]  (5);
\draw[->] (7) to [in=90,out=90]  (2);
\end{tikzpicture}
& \qquad &
\begin{tikzpicture}[baseline=0,>=stealth,scale=.75]
\foreach \x in {1,...,7}{
\coordinate (\x) at (\x,0);
}
\foreach \x in {1,2,4,5}{
 \draw (\x) node[below,fill=white,draw=black,circle,inner sep =2] {} node[below,yshift=-5pt] {$\x$};
}
\foreach \x in {3,6,7}{
 \draw (\x) node[below,fill=black,draw=black,circle,inner sep =2] {} node[below,yshift=-5pt] {$\x$};
}
\draw[->] (1) to [in=90,out=90]  (7);
\draw[->] (1) to [in=90,out=90]  (3);
\draw[->] (4) to [in=90,out=90]  (7);
\draw[->] (4) to [in=90,out=90]  (6);
\draw[->] (5) to [in=90,out=90]  (6);
\draw[->] (2) to [in=90,out=90]  (7);
\end{tikzpicture}\\
T & & T'
\end{array}
\]
\caption{Two alternating trees on 7 nodes. Tree $T$ is alternating but not positive alternating, while $T'$ is positive alternating.}\label{fig:alternatingtrees}
\end{figure}
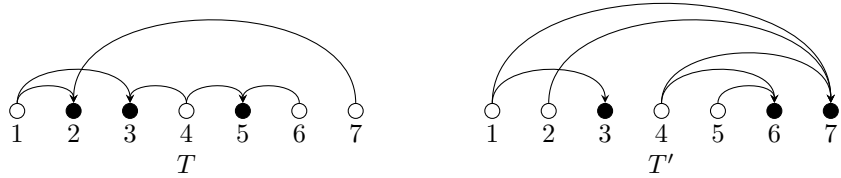

For any undirected tree on $[n]$, the identification of node 1 as a source or sink determines the direction of all arcs in an alternating tree. Thus there are precisely two alternating trees with the same undirected tree structure. As there are $n^{n-2}$ undirected trees by Cayley's Theorem, there are $2n^{n-2}$ alternating trees on $[n]$. The number of positive alternating trees on $[n]$ is:
\[
 \frac{ 1}{n2^{n-1}}\sum_{k=1}^{n} \binom{n}{k} k^{n-1},
\]
though the formula is not as simple to explain. See \cite{Chauve}.

For the purposes of the current paper, a \emph{triangulation} of a polytope $\P = \text{conv}\{0, v_1, \ldots, v_n\}$ with vertices $\{v_1, \ldots, v_n\}$ is a simplicial complex such that the union of the top dimensional simplices  of the simplicial complex is the polytope $\P$ and so that the  simplices  only use  the vectors in $\{0, v_1,\ldots,v_n\}$ as vertices. A triangulation is called \emph{central} if every top dimensional simplex contains the origin, and we call a top dimensional simplex in such a triangulation a \emph{central simplex}. In what follows we are only concerned with top dimensional simplices.

\begin{lem}(cf. \cite[Lemma 13.2]{Post09}) \label{lem:tri}  
Every top dimensional simplex in a central triangulation of $\P(\kg_n)$ is of the form 
 $\Delta_T$ for some alternating tree $T$ on the vertex set $[n]$.  Every top dimensional simplex in a central triangulation of $\P(\kg_n^+)$ is of the form 
 $\Delta_T$ for some positive alternating tree $T$ on the vertex set $[n]$.
\end{lem}

In Figure \ref{fig:rootpoly2d} we see the alternating trees that label the central triangulations of $\P(\kg_3)$ and $\P(\kg_3^+)$.

We remark that Definition \ref{def:alt} generalizes the notion of alternating introduced by Postnikov \cite[Definition 13.1]{Post09}, which was only defined for graphs where all the edges are directed from smaller to larger vertices. Lemma \ref{lem:sim} and Lemma \ref{lem:tri} generalize Lemmas 12.5 and 12.6 from Postnikov's beautiful paper \cite{Post09}. We invite the interested reader to check that Postnikov's proofs of the aforementioned lemmas readily lend themselves to generalization to our case of arbitrarily directed edges.

\subsection{Flows and root cones}

While alternating trees were designed to capture the geometry of triangulations of root polytopes, we can use the same data structure to study chamber geometry, by taking the cone over a root polytope.

\begin{defn}[Flows and root cones]\label{def:rootcone}
Suppose $G$ is a directed graph on vertex set $[n]$. A \textbf{nonnegative flow} on $G$ is a nonnegative labeling of the arcs of $G$, $f: E(G) \to \R_{\geq 0}$. Any such flow \textbf{induces a point} $\mathbf{x}=\mathbf{x}(G;f) = (x_1,\ldots,x_n) \in \R^n$ by 
\[
x_i = \sum_{(i,j) \in E(G)} f(i,j) - \sum_{(k,i)\in E(G)} f(k,i).
\]
The collection of all points induced by flows in this way make up the \textbf{root cone} for $G$:
\begin{equation}\label{eq:rootcone}
\C(G)=\left\{ \sum_{(i,j)\in E(G)} f(i,j)(e_i-e_j) : f(i,j) \geq 0 \right\}. 
\end{equation}
\end{defn}

It is easy to verify that the combinatorial structure of the faces of the  root simplices which contain the origin is the same as the combinatorial structure of the faces of the of root cones.   We now state a few properties about the geometry of root cones.

For any alternating graph $G$, the point induced by the flow satisfies
\[
 x_i = \begin{cases} \sum_{(i,j) \in E(G)} f(i,j) & \mbox{if $i$ is a source,} \\
   -\sum_{(k,i) \in E(G)} f(k,i) & \mbox{if $i$ is a sink.}
   \end{cases}
\]
In particular, the coordinates corresponding to sources are always nonnegative and the coordinates of sinks are always nonpositive. For example, below is a flow on the tree $T$ from Figure \ref{fig:alternatingtrees}:
\[
\begin{tikzpicture}[baseline=0,>=stealth,xscale=1.3]
\foreach \x in {1,...,7}{
\coordinate (\x) at (\x,0);
}
\foreach \x in {1,4,6,7}{
 \draw (\x) node[below,fill=white,draw=black,circle,inner sep =2] {} node[below,yshift=-5pt] {$\x$};
}
\foreach \x in {2,3,5}{
 \draw (\x) node[below,fill=black,draw=black,circle,inner sep =2] {} node[below,yshift=-5pt] {$\x$};
}
\draw[->] (1) to [in=90,out=90] node[midway,inner sep=2,fill=white,draw=black] {$a$} (2);
\draw[->] (1) to [in=90,out=90] node[midway,inner sep=2,fill=white,draw=black] {$b$} (3);
\draw[->] (4) to [in=90,out=90] node[midway,inner sep=2,fill=white,draw=black] {$c$} (3);
\draw[->] (4) to [in=90,out=90] node[midway,inner sep=2,fill=white,draw=black] {$d$} (5);
\draw[->] (6) to [in=90,out=90] node[midway,inner sep=2,fill=white,draw=black] {$e$} (5);
\draw[->] (7) to [in=90,out=90] node[midway,inner sep=2,fill=white,draw=black] {$f$} (2);
\end{tikzpicture}
\]
It induces the point
\[
 \mathbf{x}=(a+b,-a-f,-b-c,c+d,-d-e,e,f).
\]

Let $T$ and $T'$ be two alternating trees on $[n]$. We will be immensely interested in how the root cones $\C(T)$ and $\C(T')$ intersect. Adapting an idea of Postnikov \cite[Section 12]{Post09} helps us give a simple graph-theoretic criterion for when the interiors of two root cones overlap, which we now explain. 

Let
\[
C(T, T'):=\{(i, j): (i, j)\in E(T)\} \cup \{(i, j): (j, i)\in E(T')\},
\] following Postnikov \cite[Section 12]{Post09}.\footnote{This notation diverges slightly from Postnikov, who uses the notation $U(T,T')$ for this graph. We use $C$ to connote the word ``circulation'' as well as to avoid conflict with later notation.}
See Figure \ref{fig:UTT}, where we draw the arcs of $T$ above the nodes and the reverse of the arcs of $T'$ below.

One way to explain when two cones overlap involves the notion of a special kind of flow known as a \emph{circulation} of a directed graph $G$, i.e., a nonnegative flow that satisfies \emph{conservation of flow} at each vertex $i$:
 \[
   \sum_{(i,j) \in E(G)} f(i,j) = \sum_{(k,i) \in E(G)} f(k,i).
 \]
Note that a flow is a circulation if and only if it induces the point $(0,0,\ldots,0)$.

Circulation on an alternating graph is trivial, since all vertices are either sources or sinks, and hence the only flow to satisfy conservation of flow is the zero flow. But by considering circulation on $C(T, T')$, we can study flows for pairs of alternating trees that induce the same point.

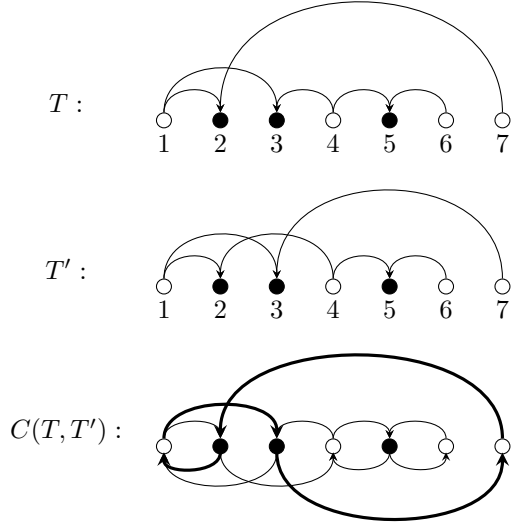
\begin{figure}
\[
\begin{array}{c c}
T: &
\begin{tikzpicture}[baseline=0,>=stealth,xscale=.75]
\foreach \x in {1,...,7}{
\coordinate (\x) at (\x,0);
}
\foreach \x in {1,4,6,7}{
 \draw (\x) node[below,fill=white,draw=black,circle,inner sep =2] {} node[below,yshift=-5pt] {$\x$};
}
\foreach \x in {2,3,5}{
 \draw (\x) node[below,fill=black,draw=black,circle,inner sep =2] {} node[below,yshift=-5pt] {$\x$};
}
\draw[->] (1) to [in=90,out=90] (2);
\draw[->] (1) to [in=90,out=90] (3);
\draw[->] (4) to [in=90,out=90] (3);
\draw[->] (4) to [in=90,out=90] (5);
\draw[->] (6) to [in=90,out=90] (5);
\draw[->] (7) to [in=90,out=90] (2);
\end{tikzpicture}\\
T': &
\begin{tikzpicture}[baseline=0,>=stealth,xscale=.75]
\foreach \x in {1,...,7}{
\coordinate (\x) at (\x,0);
}
\foreach \x in {1,4,6,7}{
 \draw (\x) node[below,fill=white,draw=black,circle,inner sep =2] {} node[below,yshift=-5pt] {$\x$};
}
\foreach \x in {2,3,5}{
 \draw (\x) node[below,fill=black,draw=black,circle,inner sep =2] {} node[below,yshift=-5pt] {$\x$};
}
\draw[->] (1) to [in=90,out=90] (2);
\draw[->] (1) to [in=90,out=90] (3);
\draw[->] (4) to [in=90,out=90] (2);
\draw[->] (4) to [in=90,out=90] (5);
\draw[->] (6) to [in=90,out=90] (5);
\draw[->] (7) to [in=90,out=90] (3);
\end{tikzpicture}\\
C(T, T'): &
\begin{tikzpicture}[baseline=0,>=stealth,scale=.75]
\foreach \x in {1,...,7}{
\coordinate (\x) at (\x,0);
}
\coordinate (1a) at (1,-.25);
\coordinate (2a) at (2,-.25);
\coordinate (3a) at (3,-.25);
\coordinate (4a) at (4,-.25);
\coordinate (5a) at (5,-.25);
\coordinate (6a) at (6,-.25);
\coordinate (7a) at (7,-.25);
\foreach \x in {1,4,6,7}{
 \draw (\x) node[below,fill=white,draw=black,circle,inner sep =2] {};
}
\foreach \x in {2,3,5}{
 \draw (\x) node[below,fill=black,draw=black,circle,inner sep =2] {};
}
\draw[->] (1) to [in=90,out=90]  (2);
\draw[very thick,->] (1) to [in=90,out=90]  (3);
\draw[->] (4) to [in=90,out=90]  (3);
\draw[->] (4) to [in=90,out=90]  (5);
\draw[->] (6) to [in=90,out=90]  (5);
\draw[very thick,->] (7) to [in=90,out=90]  (2);
\draw[very thick,->] (2a) to [in=-90,out=-90]  (1a);
\draw[->] (2a) to [in=-90,out=-90]  (4a);
\draw[very thick,->] (3a) to [in=-90,out=-90]  (7a);
\draw[->] (3a) to [in=-90,out=-90]  (1a);
\draw[->] (5a) to [in=-90,out=-90]  (4a);
\draw[->] (5a) to [in=-90,out=-90]  (6a);
\end{tikzpicture}
\end{array}
\]
\caption{The graphs $T$, $T'$, and $C(T, T')$. A directed $4$-cycle of $C(T, T')$ is highlighted in bold.}\label{fig:UTT}
\end{figure}

That is, suppose $f$ is a flow on the arcs of $T$ that induces a point $\xx$ and $g$ is another flow, this time on the arcs of $T'$, that also induces $\xx$. Then the flow $h: E(C(T, T')) \to \R_{\geq 0}$ defined by
\[
 h(i,j) = f(i,j) + g(j,i),
\]
with flows $f(i,j)=0$ if $(i,j)\notin T$ and $g(j,i)=0$ if $(j,i)\notin T'$, induces the point $\xx-\xx = 0$ and hence $h$ is a circulation. 

Conversely, any circulation on $C(T, T')$ decomposes into a flow on the arcs of $T$ that induces $\xx$ and a flow on the arcs of $T'$ that also induces $\xx$. With this observation we've already proved the first part of the following lemma.

\begin{lem}\label{lem:circulation}
Let $T$ and $T'$ be alternating trees on $[n]$. Every nonnegative circulation on $C(T, T')$ induces a point $\xx \in \C(T)\cap\C(T')$. Furthermore, $\C(T)\cap \C(T')$ is full-dimensional if and only if there is a strictly positive circulation on $C(T, T')$.
\end{lem}

\begin{proof}
To prove the second assertion, we remark that a point $\xx$ is in the interior of a root cone if and only if it is induced by a strictly positive flow. Thus a point in the interior of both $\C(T)$ and $\C(T')$ has the form
\[
 \xx = \sum_{(i,j) \in T} f(i,j)(e_i-e_j) = \sum_{(k,l) \in T'} g(k,l)(e_k-e_l),
\]
for a strictly positive flow $f$ on $T$ and a strictly positive flow $g$ on $T'$. Notice this can happen if and only if at each vertex $i$, $i$ is a source in both $T$ and $T'$ or $i$ is a sink in both $T$ and $T'$. Otherwise, if, say, $i$ was a source in $T$ and a sink in $T'$, then positivity of flow would imply $x_i >0$ for $T$ and $x_i < 0 $ for $T'$.

We can now combine these flows to form a new flow $h$ on $C(T, T')$ via
\[
 h(i,j) = \begin{cases} f(i,j) & \mbox{if $(i,j) \in E(T)$} \\
  g(i,j) &\mbox{if $(j,i) \in E(T')$}.
  \end{cases}
\]
As these cases are disjoint, we know $h$ is well-defined on $C(T, T')$. Moreover this flow induces $\xx(C(T, T');h) = \xx(T;f)-\xx(T';g) = \xx-\xx =0$, so $h$ is clearly a strictly positive circulation.

The converse is straightforward. A strictly positive circulation on $C(T, T')$ yields two strictly positive flows: one on the arcs of $T$ and another on the arcs of $T'$ (after reversing the arcs), and by conservation of flow in $C(T, T')$, both these flows induce the same point in the interior of $\C(T) \cap \C(T')$.
\end{proof}

\begin{remark}[Long cycles in $C(T, T')$]
We remark that there is a simple sufficient condition for a nontrivial circulation given in \cite[Lemma 12.6]{Post09}. We review the idea here for convenience. If, as in the Figure \ref{fig:UTT}, there is a directed $k$-cycle in $C(T, T')$ with $k\geq 4$ (and $k$ necessarily even and all edges distinct), then the arcs of the cycle above the nodes give rise to a matching $M$ on $T$, while the arcs of the cycle below the nodes give rise to a matching $M'$ on $T'$. By construction, the nodes in $M$ are the same as the nodes in $M'$, and so the point
\[
 \mathbf{x} = \sum_{(i,j) \in E(M)} (e_i-e_j) = \sum_{(j,i)\in E(M')} (e_i-e_j)
\]
is an element of $\C(T)\cap \C(T')$. The smallest face of $\C(T)$ in which $\mathbf{x}$ lives is $\C(M)$, while the smallest face of $\C(T')$ in which $\mathbf{x}$ lives is $\C(M')$. But since $M\neq M'$, we know $\C(M)\neq \C(M')$ and the intersection $\C(T)\cap \C(T')$ is not a common face. Hence,  \cite[Lemma 12.6]{Post09} concludes that alternating trees on $[n]$, such that $C(T, T')$ has no cycles of length 4 or greater, are those whose root simplices can both appear in a central triangulation of $\P(\kg_n)$.
\end{remark}

\section{The resonance arrangement}\label{sec:resonance}

In this section we present the resonance arrangement and show how certain sets of trees can be used to label chambers. First, we must properly define the resonance arrangement.

For any subset $S\subseteq [n]$, let $u_S$ denote the $0/1$ vector of length $n$ in which the elements of $S$ denote the entries that are $1$. For example, if $n=8$,
\[
 u_{\{1,3,4,6\}} = (1,0,1,1,0,1,0,0).
\]
If $S\neq \emptyset$, let $U_S = \{ \xx \in \R^{n} :\langle \xx, u_S \rangle = 0\}$ denote the hyperplane normal to $u_S$. The \emph{resonance arrangement} $\Res_{n-1}$ is the rank $n-1$ hyperplane arrangement given by the intersection of the hyperplanes $U_S$, $S\neq [n]$, with the hyperplane $V_{\emptyset}=U_{[n]} = \{\xx \in \R^n : \langle \xx, \mathbf{1} \rangle = \sum x_i =0\}$. That is, the ambient vector space for $\Res_{n-1}$ is $V_{\emptyset}$, and the hyperplanes in $\Res_{n-1}$ are given by $U'_S = U_S\cap V_{\emptyset}$. The number of chambers in $\Res_{n-1}$ is $R_{n-1}$.

For example, in Figure \ref{fig:res2} we see the resonance arrangement of rank two. Here we obtain three distinct hyperplanes (lines):
\begin{align*}
U'_1 &= \{(x,y,z)\in \R^n : x=0=y+z\},\\ 
U'_2 &= \{(x,y,z)\in \R^n : y=0=x+z\},\\
U'_{12} &= \{(x,y,z)\in \R^n : x+y=0=z\}.,
\end{align*}
corresponding to intersecting each of the following hyperplanes (planes in $\R^3$) with $V_{\emptyset}$:
\begin{align*}
U_1 &= \{(x,y,z)\in \R^n : x=0\},\\ 
U_2 &= \{(x,y,z)\in \R^n : y=0\},\\
U_{12} &= \{(x,y,z)\in \R^n : x+y=0\}.
\end{align*}
These hyperplanes have normal vectors $u_1 = (1,0,0)$, $u_2 = (0,1,0)$, and $u_{12} = (1,1,0)$.

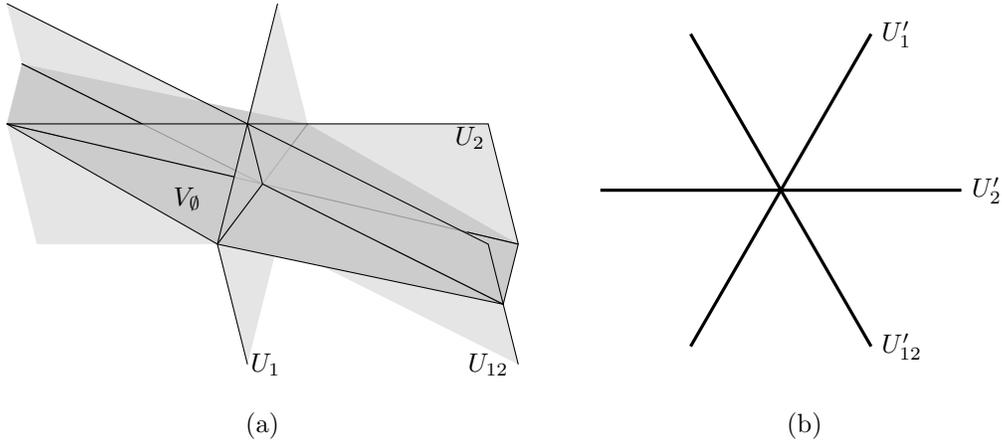
\begin{figure}
\begin{center}
\begin{tabular*}{\columnwidth}{@{\extracolsep{\fill}} cccc}
&
\begin{tikzpicture}[scale=.8,baseline=0]
\draw[draw=none,fill=white!80!black,opacity=1] (-4,1)--(-3.75,2)--(1,1)--(4.5,-1)--(4.25,-2)--(-.5,-1)--(-4,1);
\draw (-4,1)--(4.5,-1);
\draw (-3.75,2)--(4.25,-2);
\draw (1,1)--(-.5,-1);
\draw[draw=none, fill=white!80!black,opacity=.5] (-4,3)--(4,-1)--(4.5,-3)--(-3.5,1)--(-4,3);
\draw[draw=none, fill=white!80!black,opacity=.5] (-4,1)--(4,1)--(4.5,-1)--(-3.5,-1)--(-4,1);
\draw[draw=none, fill=white!80!black,opacity=.5] (.5,3)--(1,1)--(0,-3)--(-.5,-1)--(.5,3);
\draw (-4,3)--(4,-1)--(4.5,-3);
\draw (-4,1)--(4,1)--(4.5,-1);
\draw (0,-3)--(-.5,-1)--(.5,3);
\draw (0,1)--(.25,0);
\draw (-.5,-1)--(.25,0)--(4.25,-2);
\draw (-4,1)--(-.5,-1)--(4.25,-2)--(4.5,-1);
\draw (-4,1)--(-.22,.12);
\draw (4.5,-1)--(3.65,-.8);
\draw (-3.75,2)--(-1.75,1);
\draw (0,-3) node[right, fill=none,inner sep=1] {$U_1$};
\draw (4,-3) node[fill=none,inner sep=1] {$U_{12}$};
\draw (4,1) node[below left,fill=none,inner sep=1] {$U_2$};
\draw (-1,-.5) node[above,fill=none,inner sep=1] {$V_{\emptyset}$};
\end{tikzpicture}
&
\begin{tikzpicture}[scale=1.2, cm={1,0,.5,.8660254,(0,0)}, >=stealth,baseline=3.2cm]
      \draw[very thick] (0,3)--(4,3) node[right] {$U'_{2}$};
      \draw[very thick] (2,1)--(2,5) node[right] {$U'_{1}$};
      \draw[very thick] (0,5)--(4,1) node[right] {$U'_{12}$};
\end{tikzpicture}
&
\\
\\
& (a) & (b) &
\end{tabular*}
\end{center}
\caption{The rank two resonance arrangement, $\Res_2$, drawn (a) in $\R^3$, and (b) in the plane $V_{\emptyset}= \{ (x,y,z) : x+y+z=0\}$.}\label{fig:res2}
\end{figure}

In Figure \ref{fig:rank3res} we see an image of the rank 3 resonance arrangement.

\subsection{Sign vectors for the resonance arrangement}

It turns out that while there are $2^{n}$ vectors $u_S$, more than half of them are not important for characterizing chambers. We can immediately discard $u_{\emptyset}$, since it is the zero vector, and $u_{[n]}= (1,1,\ldots,1)$ is normal to all of $V_{\emptyset}$. Of the remaining $2^{n}-2$ hyperplanes, we note that $U_S'=U_{[n]-S}'$ for any $S\subseteq [n],$ so we can discard proper, nonempty subsets with $n\in S.$ This yields $(2^n-2)/2 = 2^{n-1}-1$ entries that determine a sign vector for a point in $\Res_{n-1}$, indexed by nonempty subsets of $[n-1]$.

\begin{defn}[Resonance sign vectors]\label{def:resvector}
Given a point $\xx \in V_{\emptyset} = \{ \xx \in \R^n : \sum x_i =0\}$, the \textbf{resonance sign vector} is given by
\[
 \sigma(\xx) = (\sigma_S(\xx))_{\emptyset \subsetneq S \subseteq [n-1]},
\]
where 
\[
\sigma_S(\xx) = \begin{cases} + & \mbox{if } \langle \xx, u_S \rangle > 0,\\
 - & \mbox{if } \langle \xx, u_S \rangle < 0,\\
 0 & \mbox{if } \langle \xx, u_S \rangle = 0.
 \end{cases}
\]
\end{defn}

For example, the point $\xx = (1,-2,0,1) \in V_{\emptyset}$ has $\sigma(\xx)$ given by
\[
 (\sigma_1, \sigma_2, \sigma_3, \sigma_{12}, \sigma_{13}, \sigma_{23}, \sigma_{123}) = (+,-,0,-,+,-,-).
\]

\begin{remark}[Subset sums]
The inner product $\langle \xx, u_S \rangle = \sum_{s \in S} x_s$ records the sign of the sum of the entries indexed by $S$. Determining whether a point lies in the interior of a chamber of the resonance arrangement then amounts to the NP-complete problem \texttt{SubsetSum}, which asks ``given a multiset of integers, does it have a subset with sum 0?'' However, the complexity of this problem in general is not a hindrance for computing sign vectors in relatively small examples. A similar observation was made in Appendix A of \cite{Lewis17}.
\end{remark}

\subsection{Sign vectors for trees}

Not all points in the interior of a root cone---even the root cone of a tree---have the same sign vector. For example, both the points $(2,2,-3,-1)$ and $(1,3,-2,-2)$ lie in the root cone for the same tree, induced by the flows shown in Figure \ref{fig:twosigns}. However, $\sigma_{23}((2,2,-3,-1)) = -$ while $\sigma_{23}((1,3,-2,-2)) = +$ and so these two points are separated by the hyperplane $U_{23}'$. Incidentally, this means we can construct a third point in this root cone, on the line between these two, with $\sigma_{23} = 0$. While we have an entry of the sign vectors that disagree, one can check that all other entries are the same. 

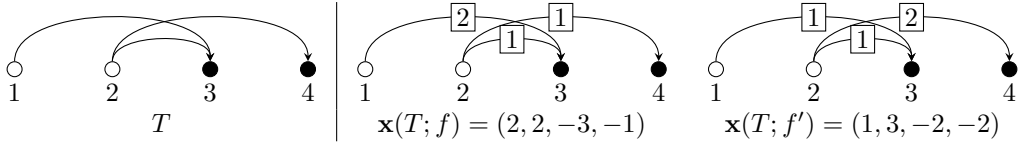
\begin{figure}
\[
\begin{array}{c | c c}
\begin{tikzpicture}[baseline=0,>=stealth,xscale=1.3]
\foreach \x in {1,...,4}{
\coordinate (\x) at (\x,0);
}
\foreach \x in {1,2}{
 \draw (\x) node[below,fill=white,draw=black,circle,inner sep =2] {} node[below,yshift=-5pt] {$\x$};
}
\foreach \x in {3,4}{
 \draw (\x) node[below,fill=black,draw=black,circle,inner sep =2] {} node[below,yshift=-5pt] {$\x$};
}
\draw[->] (1) to [in=90,out=90]  (3);
\draw[->] (2) to [in=90,out=90]  (3);
\draw[->] (2) to [in=90,out=90]  (4);
\end{tikzpicture}
&
\begin{tikzpicture}[baseline=0,>=stealth,xscale=1.3]
\foreach \x in {1,...,4}{
\coordinate (\x) at (\x,0);
}
\foreach \x in {1,2}{
 \draw (\x) node[below,fill=white,draw=black,circle,inner sep =2] {} node[below,yshift=-5pt] {$\x$};
}
\foreach \x in {3,4}{
 \draw (\x) node[below,fill=black,draw=black,circle,inner sep =2] {} node[below,yshift=-5pt] {$\x$};
}
\draw[->] (1) to [in=90,out=90] node[midway,inner sep=2,fill=white,draw=black] {$2$} (3);
\draw[->] (2) to [in=90,out=90] node[midway,inner sep=2,fill=white,draw=black] {$1$} (3);
\draw[->] (2) to [in=90,out=90] node[midway,inner sep=2,fill=white,draw=black] {$1$} (4);
\end{tikzpicture}
&
 \begin{tikzpicture}[baseline=0,>=stealth,xscale=1.3]
\foreach \x in {1,...,4}{
\coordinate (\x) at (\x,0);
}
\foreach \x in {1,2}{
 \draw (\x) node[below,fill=white,draw=black,circle,inner sep =2] {} node[below,yshift=-5pt] {$\x$};
}
\foreach \x in {3,4}{
 \draw (\x) node[below,fill=black,draw=black,circle,inner sep =2] {} node[below,yshift=-5pt] {$\x$};
}
\draw[->] (1) to [in=90,out=90] node[midway,inner sep=2,fill=white,draw=black] {$1$} (3);
\draw[->] (2) to [in=90,out=90] node[midway,inner sep=2,fill=white,draw=black] {$1$} (3);
\draw[->] (2) to [in=90,out=90] node[midway,inner sep=2,fill=white,draw=black] {$2$} (4);
\end{tikzpicture}
\\
T & \xx(T;f) = (2,2,-3,-1) & \xx(T;f')=(1,3,-2,-2)
\end{array}
\]
\caption{Two points in the same root cone with different sign vectors.}\label{fig:twosigns}
\end{figure}

The first result in this section is a lemma that tells us something about the extent to which the sign vector of a point in a root cone is determined by the graph itself, rather than a particular flow on the graph. 

For any directed graph $G$ on $[n]$, let $G_I$ denote the subgraph restricted to the vertices in $I \subset [n]$. We let $\ind(G_I)$ denote the set of arcs entering $G_I$, and we say $|\ind(G_I)|$ is the \emph{indegree of subset $I$}. Similarly we let $\outd(G_I)$ denote the set of arcs leaving $G_I$, and we call $|\outd(G_I)|$ the \emph{outdegree of subset $I$}. To be clear,
\[
\ind(G_I) = \{ (j,i) : i \in I, j \notin I\},
\]
and
\[
\outd(G_I) = \{(i,j) : i \in I, j \notin I\}.
\]
The following lemma shows that sometimes the indegree and outdegree are sufficient to determine entries in the sign vector of a point in $\C(G)$.

\begin{lem}\label{lem:signflow}
 Let $G$ be an alternating acyclic graph on $[n]$. Let $I$ be any nonempty subset of $[n-1]$, i.e., $\emptyset \neq I \subseteq [n-1]$. Then, for any point $\xx \in \C(G)$: 
 \begin{itemize}
  \item if $|\ind(G_I)|=0$, then $\sigma_I(\xx) \in \{+,0\}$, or
  \item if $|\outd(G_I)|=0$, then $\sigma_I(\xx) \in \{-,0\}$.
 \end{itemize}
In particular if $G_I$ is not connected to $G_{[n]-I}$ then $\sigma_I(\xx) = 0$. 
\end{lem}

\begin{proof}
The two cases are identical up to reversal of the arcs of $G$, so without loss of generality, suppose $\ind(G_I)=\emptyset$. We will show $\sigma_I(\xx)=+$ or $0$. 
 
We can partition the arcs of $G$ into three sets $E_I = \{ (i,j) : i,j \in I\}$, $E_{[n]-I} = \{ (k,l) : k, l \in [n]-I\}$, and $\outd(G_I) = \{ (i,l) : i \in I, l \in [n]-I\}$.
 
Let $f$ be a nonnegative flow on $G$, and let $\xx=\xx(G;f)$ be the point induced by this flow. We have 
\[
\langle \xx, u_I \rangle = \sum_{i\in I} x_i = \sum_{(i,k)\in \outd(G_I)} f(i,k),
\] 
since for every arc $(i,j)$ with both $i,j \in I$, we have $f(i,j)$ contributing to $x_i$ and $-f(i,j)$ contributing to $x_j$. Since $f$ is a nonnegative flow, we have $\sigma_I(\xx) = +$ if any of the flows $f(i,k) \in \outd(G_I)$ are positive and $\sigma_I(\xx)=0$ if all the flows are zero (or $\outd(G_I)$ is empty). 
\end{proof}

We use the idea of Lemma \ref{lem:signflow} to define a coarse sort of sign vector for root cones themselves, i.e., for alternating trees.

\begin{defn}[Tree sign vector]\label{def:treevector}
Let $T$ be an alternating tree on $[n]$. The \textbf{tree sign vector} is given by
\[
\sigma(T) = (\sigma_S(T))_{\emptyset \subsetneq S \subseteq [n-1]},
\]
where 
\[
\sigma_S(T) = \begin{cases} + & \mbox{if } |\ind(T_S)| = 0,\\
 - &  \mbox{if } |\outd(T_S)| = 0,\\
 ? & \mbox{otherwise}.
 \end{cases}
\]
\end{defn}

In particular, we have $\sigma_I(T) =+$ for any subset of sources $I$ and $\sigma_J(T) =-$ for any subset of sinks $J$. (The ``interesting'' entries are those $\sigma_S(T)$ for which $S$ contains both sources and sinks.) For example, the tree in Figure \ref{fig:twosigns} has $\sigma(T)$ given by
\[
 (\sigma_1, \sigma_2, \sigma_3, \sigma_{12}, \sigma_{13}, \sigma_{23}, \sigma_{123}) = (+,+,-,+,-,?,+).
\]

It will be good to know when two trees have nearly the same sign vectors.

\begin{defn}[Sign compatible trees]\label{def:signcompatible}
Let $T$ and $T'$ be two alternating trees on $[n]$. We say the trees are \textbf{sign incompatible} if they disagree on a known entry in their sign vectors, i.e., if for some $I$, $\{ \sigma_I(T), \sigma_I(T')\} = \{ -,+\}$. Otherwise, we say $T$ and $T'$ are \textbf{sign compatible}.
\end{defn}

In other words, two trees are sign compatible if for all $I$, either $\sigma_I(T) = \sigma_I(T')$, or if not equal, one of the entries is a ``$?$''.

While it is rather obvious that two root cones with full dimensional intersection must be sign compatible, it turns out that the converse is true as well. To prove this result, we invoke \emph{Hoffman's circulation theorem}, as stated here.

\begin{thm}[Hoffman's circulation theorem \cite{SchrijverA, Hof03}]\label{thm:Hoffman}  Let $G=(V, E)$ be a directed graph, and suppose there exist flows $l$ and $u$ on $G$, with  $0\leq l(i,j) \leq u(i,j)$ for each $(i,j) \in E$. Then there exists a circulation $f$ on $G$ with $l(i,j)\leq f(i,j) \leq u(i,j)$ for all $(i,j) \in E$ if and only if 
\[
\sum_{(i,j) \in \ind(G_I)} l(i,j) \leq \sum_{(k,l)\in \outd(G_I)}u(i,j)
\]
for all $\emptyset \neq I\subsetneq V$. 
\end{thm} 

We will also have use for the following lemma.

\begin{lem}\label{lem:outU}
Suppose $T$ and $T'$ are alternating trees on $[n]$, and let $C = C(T, T')$. If $T$ and $T'$ are sign compatible, then for all $\emptyset \neq I \subsetneq [n]$, we have $\ind(C_I) \neq \emptyset$ and $\outd(C_I) \neq \emptyset$.
\end{lem}

\begin{proof}
We will prove the contrapositive. Suppose $\outd(C_I) = \emptyset$ for some $I$. (The argument for $\ind(C_I)=\emptyset$ is similar.) Then by construction of $C=C(T, T')$, we would have $\outd(T_I) = \emptyset$ and $\ind(T'_I) = \emptyset$. By definition of the tree sign vector, this means $\sigma_I(T)=-$ and $\sigma_I(T') = +$ and $T$ and $T'$ are sign incompatible.
\end{proof}

The following result uses Hoffman's circulation theorem to show that sign compatibility for a pair of trees is equivalent to full-dimensional intersection of root cones.

\begin{cor} \label{cor:signcompatible} Let $T$ and $T'$ be two alternating trees on $[n]$. Then the intersection of $\C(T)$ and $\C(T')$ is full-dimensional, i.e., $\dim(\C(T)\cap \C(T')) = n-1$, if and only if $T$ and $T'$ are sign compatible.
\end{cor}

\begin{proof}
If the intersection of $\C(T)$ and $\C(T')$ is full-dimensional, then the sign-compatibility of $T$ and $T'$ is obvious, since there exists a point in the interior of both cones.

Now suppose $T$ and $T'$ are sign compatible. Then we claim Hoffman's circulation theorem is satisfied by letting $l(i,j) = 1$ and $u(i,j) = 2n$ for all arcs $(i,j)$ in $C=C(T, T')$. Indeed, by Lemma \ref{lem:outU}, we know $\outd(C_I) \neq \emptyset$ for any proper nonempty subset $I$. Thus,
\[
\sum_{(i,j) \in \ind(C_I)} l(i,j) = |\ind(C_I)| \leq |E(C)| = 2(n-1) < 2n \leq 2n|\outd(C_I)| = \sum_{(i,j) \in \outd(C_I)} u(i,j).
\]
Hence, there exists a circulation $f: E(C) \to \R$ where $f(i,j)\geq l(i,j)=1$ on every arc $(i,j)\in E(C).$  That is, $f$ is strictly positive.  Lemma \ref{lem:circulation} now implies that  $\C(T)\cap\C(T')$ is full dimensional.  
\end{proof}

\subsection{Resonance chambers as intersections of cones}

In this section we justify the fundamental connection between root cones and the resonance arrangement. Given a collection of cones with full-dimensional intersection, we call the interior of their intersection a \emph{refined chamber}. Refined chambers are ordered by reverse inclusion, and a \emph{maximally refined chamber} is a refined chamber that does not contain any other refined chambers.

\begin{prop}\label{lem:chambersandroots}
 The chambers of the resonance arrangement are the maximally refined chambers obtained by intersections of root cones.
\end{prop}

\begin{proof}
It will suffice to argue the complement: that the union of the walls of the root cones is precisely the resonance hyperplane arrangement.
 
We first show that any wall of a root cone lies in a hyperplane of the resonance arrangement. By Lemma \ref{lem:sim}, a wall in a root cone $\C(T)$ is itself a root cone for an acyclic graph with $n-2$ edges, i.e., a disjoint union of two trees $T_I\cup T_{[n]-I}$, with no edges connecting a vertex in $I$ with a vertex in $[n]-I$. Thus by Lemma \ref{lem:signflow}, we have that if $\xx \in \C(T_I \cup T_{[n]-I})$, $\sigma_I(\xx) = 0$. Therefore,
 \[
  \C(T_I \cup T_{[n]-I}) \subset U'_I.
 \]
 
Now we wish to show that for any point $\xx$ in a resonance hyperplane $U'_I$, there is a tree $T$ for which $\xx$ is on the boundary of $\C(T)$. But this is just to say that $\xx$ is in a root cone $\C(G)$ for some acyclic alternating graph with at most $n-2$ edges. Such a graph is easily constructed. 

Since $\xx \in U'_I$, we know in that both $\sum_{i \in I} x_i = 0$ and $\sum_{j \in [n]-I} x_j =0$. Consider the orthogonal pair of points $\xx_I = \sum_{i \in I} x_ie_i$, and $\xx_{[n]-I} = \sum_{j \in [n]-I} x_je_j$. We see $\xx_I$ lives in an $(|I|-1)$-dimensional subspace, and hence it is in the cone of some acyclic graph $G_I$. Likewise $\xx_{[n]-I}$ is induced by a graph $G_{[n]-I}$, and their sum, $\xx$, is induced by their disjoint union:
\[
 \xx \in \C( G_I \cup G_{[n]-I} ).
\] 
\end{proof}

\begin{defn}[Indexable collections]\label{def:indexable}
Let $\textbf{T}=\{T_1,\ldots,T_k\}$ be a set of pairwise sign compatible alternating trees on $[n]$. If the set of points simultaneously induced by each tree in the collection is full-dimensional, i.e., if $\dim(\C(T_1)\cap \cdots \cap \C(T_k)) = n-1$, we say $\textbf{T}$ is \textbf{indexable}.
\end{defn}

In other words, an indexable collection of trees corresponds to a collection of root cones whose intersection is a refined chamber. Since Proposition \ref{lem:chambersandroots} says that resonance chambers are maximally refined chambers, we let $\In_n$ denote the set of maximal (under inclusion) indexable collections of alternating trees on $[n]$.

By Proposition \ref{lem:chambersandroots}, then, the number of chambers in the resonance arrangement is the same as the number of maximal indexable collections.

\begin{cor}\label{cor:indexableR}
The number of maximal indexable collections of trees on $[n+1]$ equals the number of chambers in the $n$-dimensional resonance arrangement, i.e.,
\[
 R_n = |\In_{n+1}|.
\]
\end{cor}

There is a rather nice symmetry of the resonance arrangement $\Res_n$ given by cyclic permutation of coordinates, which we now discuss. Let $\omega: \R^{n+1} \to \R^{n+1}$ be the cyclic permutation of the standard basis given by
\[
 \omega \xx = (x_2,\ldots,x_{n+1},x_1).
\]
Now consider the action of $\omega$ on the full positive root cone:
\[
\C=\C(\kg_{n+1}^+)= \left\{ \sum_{1\leq i<j \leq n+1} f(i,j)(e_i-e_j) : f(i,j) \geq 0 \right \}.
\]

The following lemma is well known and we leave the proof of it as an exercise for the reader.

\begin{lem}\label{lem:cyclicC}
The cones $\C, \omega \C, \omega^2 \C, \ldots, \omega^n \C$ have pairwise disjoint interiors and their union is all of $V_{\emptyset}$. 
\end{lem}
 
Each of the cones in Lemma \ref{lem:cyclicC} contains the same number of chambers as the resonance arrangement. We record this observation as follows, where we let $R_n^+$ denote the number of resonance chambers in the positive root cone $\C(\kg_{n+1}^+)$.

\begin{cor}\label{cor:cyclicsymmetry}
The number of resonance chambers in $\Res_n$ is equal to $(n+1)$ times the number of resonance chambers in the positive root cone $\C(\kg_{n+1}^+)$: $R_n = (n+1)R_n^+$. In particular,
\[
 R_n^+ = \frac{|\In_{n+1}|}{(n+1)}.
\]
\end{cor}

In Figure \ref{fig:rootpoly2d} we see $\Res_2$ (actually the polytope $\P(\kg_3)$) labeled by alternating trees, and in Figure \ref{fig:labeledrank3}, we see $32/4=8$ chambers of $\Res_3$ that lie in the positive root cone $\C(\kg_4^+)$, labeled by indexable collections on $[4]$. 

\begin{figure}
\[
\begin{tikzpicture}[scale=2.5]
\coordinate (a) at (0,4);
\coordinate (b) at (2,0);
\coordinate (c) at (-2,0);
\coordinate (d) at (1.55,2.36);
\coordinate (e) at (-1.55,2.36);
\coordinate (f) at (0,1.2);
\coordinate (g) at (0,-.4);
\draw (a) to [in=90,out=-30]  (b);
\draw (b) to [in=-20,out=200]  (c);
\draw (c) to [in=210,out=90]  (a);
\draw (c) to [in=220,out=30] (d);
\draw (b) to [in=-40,out=150] (e);
\draw (a) to (g);
\draw (d) to [in=10,out=170] (e);
\draw (a) node[fill=black,circle, inner sep=2] {} node[above] {$e_2-e_3$};
\draw (b) node[fill=black,circle, inner sep=2] {} node[right] {$e_3-e_4$};
\draw (c) node[fill=black,circle, inner sep=2] {} node[left] {$e_1-e_2$};
\draw (d) node[fill=black,circle, inner sep=2] {} node[above right] {$e_2-e_4$};
\draw (e) node[fill=black,circle, inner sep=2] {} node[above left] {$e_1-e_3$};
\draw (f) node[fill=black,circle, inner sep=2] {} node[left,xshift=-3pt] {$e_1-e_4$};
\draw (-.6,.25) node {
\begin{tabular}{c}
\begin{tikzpicture}[baseline=0,>=stealth,scale=.5]
\foreach \x in {1,...,4}{
\coordinate (\x) at (\x,0);
}
\foreach \x in {1,3}{
 \draw (\x) node[below,fill=white,draw=black,circle,inner sep =2] {};
}
\foreach \x in {2,4}{
 \draw (\x) node[below,fill=black,draw=black,circle,inner sep =2] {};
}
\draw[->] (1) to [in=90,out=90] (2);
\draw[->] (1) to [in=90,out=90] (4);
\draw[->] (3) to [in=90,out=90] (4);
\end{tikzpicture}\\
\begin{tikzpicture}[baseline=0,>=stealth,scale=.5]
\foreach \x in {1,...,4}{
\coordinate (\x) at (\x,0);
}
\foreach \x in {1,3}{
 \draw (\x) node[below,fill=white,draw=black,circle,inner sep =2] {};
}
\foreach \x in {2,4}{
 \draw (\x) node[below,fill=black,draw=black,circle,inner sep =2] {};
}
\draw[->] (1) to [in=90,out=90] (2);
\draw[->] (1) to [in=90,out=90] (4);
\draw[->] (3) to [in=90,out=90] (2);
\end{tikzpicture}
\end{tabular}
};
\draw (.6,.25) node {
\begin{tabular}{c}
\begin{tikzpicture}[baseline=0,>=stealth,scale=.5]
\foreach \x in {1,...,4}{
\coordinate (\x) at (\x,0);
}
\foreach \x in {1,3}{
 \draw (\x) node[below,fill=white,draw=black,circle,inner sep =2] {};
}
\foreach \x in {2,4}{
 \draw (\x) node[below,fill=black,draw=black,circle,inner sep =2] {};
}
\draw[->] (1) to [in=90,out=90] (2);
\draw[->] (1) to [in=90,out=90] (4);
\draw[->] (3) to [in=90,out=90] (4);
\end{tikzpicture}
\\
\begin{tikzpicture}[baseline=0,>=stealth,scale=.5]
\foreach \x in {1,...,4}{
\coordinate (\x) at (\x,0);
}
\foreach \x in {1,3}{
 \draw (\x) node[below,fill=white,draw=black,circle,inner sep =2] {};
}
\foreach \x in {2,4}{
 \draw (\x) node[below,fill=black,draw=black,circle,inner sep =2] {};
}
\draw[->] (3) to [in=90,out=90] (2);
\draw[->] (1) to [in=90,out=90] (4);
\draw[->] (3) to [in=90,out=90] (4);
\end{tikzpicture}
\end{tabular}
};
\draw (1.2,1.2) node {
\begin{tikzpicture}[baseline=0,>=stealth,scale=.5]
\foreach \x in {1,...,4}{
\coordinate (\x) at (\x,0);
}
\foreach \x in {1,2,3}{
 \draw (\x) node[below,fill=white,draw=black,circle,inner sep =2] {};
}
\foreach \x in {4}{
 \draw (\x) node[below,fill=black,draw=black,circle,inner sep =2] {};
}
\draw[->] (1) to [in=90,out=90] (4);
\draw[->] (2) to [in=90,out=90] (4);
\draw[->] (3) to [in=90,out=90] (4);
\end{tikzpicture}
};
\draw (.45,2.15) node {
\begin{tabular}{c}
\begin{tikzpicture}[baseline=0,>=stealth,scale=.5]
\foreach \x in {1,...,4}{
\coordinate (\x) at (\x,0);
}
\foreach \x in {1,2}{
 \draw (\x) node[below,fill=white,draw=black,circle,inner sep =2] {};
}
\foreach \x in {3,4}{
 \draw (\x) node[below,fill=black,draw=black,circle,inner sep =2] {};
}
\draw[->] (1) to [in=90,out=90] (4);
\draw[->] (2) to [in=90,out=90] (3);
\draw[->] (2) to [in=90,out=90] (4);
\end{tikzpicture}
\\
\begin{tikzpicture}[baseline=0,>=stealth,scale=.5]
\foreach \x in {1,...,4}{
\coordinate (\x) at (\x,0);
}
\foreach \x in {1,2}{
 \draw (\x) node[below,fill=white,draw=black,circle,inner sep =2] {};
}
\foreach \x in {3,4}{
 \draw (\x) node[below,fill=black,draw=black,circle,inner sep =2] {};
}
\draw[->] (1) to [in=90,out=90] (3);
\draw[->] (1) to [in=90,out=90] (4);
\draw[->] (2) to [in=90,out=90] (4);
\end{tikzpicture}
\end{tabular}
};
\draw (.5,3) node {
\begin{tabular}{c}
\begin{tikzpicture}[baseline=0,>=stealth,scale=.5]
\foreach \x in {1,...,4}{
\coordinate (\x) at (\x,0);
}
\foreach \x in {1,2}{
 \draw (\x) node[below,fill=white,draw=black,circle,inner sep =2] {};
}
\foreach \x in {3,4}{
 \draw (\x) node[below,fill=black,draw=black,circle,inner sep =2] {};
}
\draw[->] (1) to [in=90,out=90] (4);
\draw[->] (2) to [in=90,out=90] (3);
\draw[->] (2) to [in=90,out=90] (4);
\end{tikzpicture}
\\
\begin{tikzpicture}[baseline=0,>=stealth,scale=.5]
\foreach \x in {1,...,4}{
\coordinate (\x) at (\x,0);
}
\foreach \x in {1,2}{
 \draw (\x) node[below,fill=white,draw=black,circle,inner sep =2] {};
}
\foreach \x in {3,4}{
 \draw (\x) node[below,fill=black,draw=black,circle,inner sep =2] {};
}
\draw[->] (1) to [in=90,out=90] (3);
\draw[->] (2) to [in=90,out=90] (3);
\draw[->] (2) to [in=90,out=90] (4);
\end{tikzpicture}
\end{tabular}
};
\draw (-.5,3) node {
\begin{tabular}{c}
\begin{tikzpicture}[baseline=0,>=stealth,scale=.5]
\foreach \x in {1,...,4}{
\coordinate (\x) at (\x,0);
}
\foreach \x in {1,2}{
 \draw (\x) node[below,fill=white,draw=black,circle,inner sep =2] {};
}
\foreach \x in {3,4}{
 \draw (\x) node[below,fill=black,draw=black,circle,inner sep =2] {};
}
\draw[->] (1) to [in=90,out=90] (4);
\draw[->] (1) to [in=90,out=90] (3);
\draw[->] (2) to [in=90,out=90] (3);
\end{tikzpicture}
\\
\begin{tikzpicture}[baseline=0,>=stealth,scale=.5]
\foreach \x in {1,...,4}{
\coordinate (\x) at (\x,0);
}
\foreach \x in {1,2}{
 \draw (\x) node[below,fill=white,draw=black,circle,inner sep =2] {};
}
\foreach \x in {3,4}{
 \draw (\x) node[below,fill=black,draw=black,circle,inner sep =2] {};
}
\draw[->] (1) to [in=90,out=90] (3);
\draw[->] (2) to [in=90,out=90] (3);
\draw[->] (2) to [in=90,out=90] (4);
\end{tikzpicture}
\end{tabular}
};
\draw (-.45,2.15) node {
\begin{tabular}{c}
\begin{tikzpicture}[baseline=0,>=stealth,scale=.5]
\foreach \x in {1,...,4}{
\coordinate (\x) at (\x,0);
}
\foreach \x in {1,2}{
 \draw (\x) node[below,fill=white,draw=black,circle,inner sep =2] {};
}
\foreach \x in {3,4}{
 \draw (\x) node[below,fill=black,draw=black,circle,inner sep =2] {};
}
\draw[->] (1) to [in=90,out=90] (4);
\draw[->] (1) to [in=90,out=90] (3);
\draw[->] (2) to [in=90,out=90] (3);
\end{tikzpicture}
\\
\begin{tikzpicture}[baseline=0,>=stealth,scale=.5]
\foreach \x in {1,...,4}{
\coordinate (\x) at (\x,0);
}
\foreach \x in {1,2}{
 \draw (\x) node[below,fill=white,draw=black,circle,inner sep =2] {};
}
\foreach \x in {3,4}{
 \draw (\x) node[below,fill=black,draw=black,circle,inner sep =2] {};
}
\draw[->] (1) to [in=90,out=90] (3);
\draw[->] (1) to [in=90,out=90] (4);
\draw[->] (2) to [in=90,out=90] (4);
\end{tikzpicture}
\end{tabular}
};
\draw (-1.2,1.2) node {
\begin{tikzpicture}[baseline=0,>=stealth,scale=.5]
\foreach \x in {1,...,4}{
\coordinate (\x) at (\x,0);
}
\foreach \x in {1}{
 \draw (\x) node[below,fill=white,draw=black,circle,inner sep =2] {};
}
\foreach \x in {2,3,4}{
 \draw (\x) node[below,fill=black,draw=black,circle,inner sep =2] {};
}
\draw[->] (1) to [in=90,out=90] (2);
\draw[->] (1) to [in=90,out=90] (3);
\draw[->] (1) to [in=90,out=90] (4);
\end{tikzpicture}
};
\end{tikzpicture}
\]
\caption{The 8 chambers of the resonance arrangement that lie in the postive root cone, labeled by indexable collections of alternating trees on $[4]$. Labelings of all 32 chambers appear in four such groups of 8 obtained by cyclic permutation of coordinates, i.e., cyclic permutation of the nodes in the trees.}\label{fig:labeledrank3}
\end{figure}
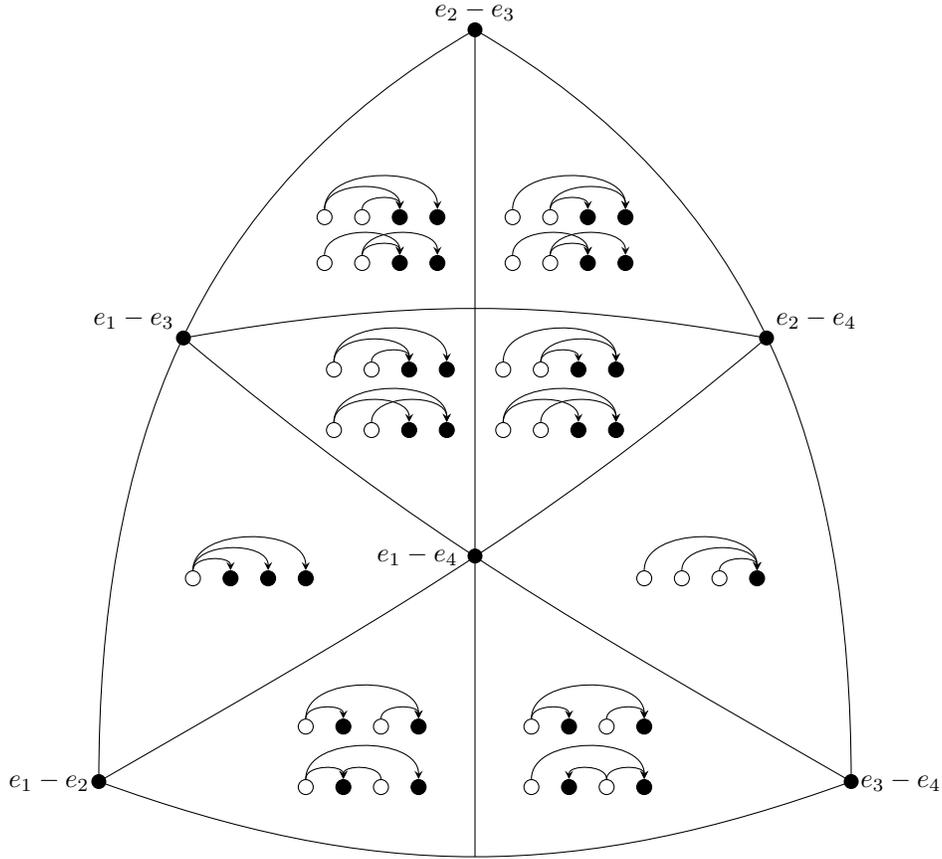
 
As discussed in the introduction, resonance chambers correspond to ``maximal unbalanced families,'' and here we have a maximal collection of trees satisfying a certain geometric condition. One may wonder whether there is a simple, direct link between maximal indexable collections and maximal unbalanced families. We do not know of such a link, but it seems worthy of investigation.

\begin{p}
Find a direct bijection between maximal unbalanced families and maximal indexable collections of alternating trees.
\end{p}

\subsection{The compatibility graph for alternating trees}\label{sec:graph}

We now use sign compatibility to define a graph whose vertices are alternating trees, with two trees adjacent if and only if they are sign compatible.

\begin{defn}[Compatibility graph]\label{def:compatiblegraph}
Define the \textbf{compatibility graph} $\Gamma_n = (V_n,E_n)$ for alternating trees on $[n]$ by 
\[
 V_n = \{ \mbox{ alternating trees $T$ on $[n]$ }\},
\]
and
\[
 E_n = \{ (T,T') : \mbox{ $T$ and $T'$ are sign compatible } \}.
\]
\end{defn}

Any indexable collection is a clique in $\Gamma_n$, but the converse is not generally true. That is, there exist pairwise sign compatible trees (i.e., with pairwise full-dimensional intersection) whose intersection is not full-dimensional, as the next example shows. 

\begin{exam}
The trees shown in Figure \ref{fig:badtriple} are pairwise sign compatible, but their mutual intersection is not full-dimensional. To see this, let $f_1$, $f_2$, and $f_3$ denote nonnegative flows on $T_1$, $T_2$, and $T_3$ respectively, such that 
\[
 \xx(T_1;f_1) = \xx(T_2;f_2) = \xx(T_3;f_3).
\]
This gives a system of linear equations for the flows $f_*(i,j)$, which includes the following relations
\begin{align*}
 f_1(1,4) &=-f_1(1,5)+f_2(1,6) \\
 f_2(2,6)+f_2(3,6) &= f_1(3,6)-f_2(1,6) \\
 0 &= -f_1(3,6) -f_1(3,5)+f_3(3,5) \\
 f_3(1,5)  &= f_1(1,5)+f_1(3,5)-f_3(3,5) 
\end{align*}
and summing, we find
\[
f_1(1,4)+f_2(2,6)+f_2(3,6) + f_3(1,5) = 0.
\]
Since the flows $f_1$, $f_2$, and $f_3$ are nonnegative, this implies that $f_1(1,4)=f_2(2,6)=f_2(3,6)=f_3(1,5) = 0$. In particular $\xx$ is not induced by a positive flow on any of the trees and $\C(T_1)\cap \C(T_2) \cap \C(T_3)$ is not full-dimensional. In fact, having flow 0 on two arcs of $T_2$ implies this intersection is at most 3-dimensional.

We can also see that the mutual intersection of all three trees is not full dimensional through considering sign vectors.  In particular
$$\sigma_{\{2, 3, 4, 5\}}(T_2)=+$$ so that for any $\xx\in \C(T_2),$ the resonance sign vector $\sigma_{\{2, 3, 4, 5\}}(\xx)\in\{0, +\}.$ However, $$\sigma_{\{2, 4\}}(T_1)=\sigma_{\{3, 5\}}(T_3)=-,$$ which implies that the resonance sign vector of any point $\xx\in \C(T_1)\cap\C(T_3)$ has $\sigma_{\{ 2, 3, 4, 5\}}(\xx)\in \{0, -\}.$
\end{exam}

\begin{figure}
\[
\begin{array}{c c c}
\begin{tikzpicture}[baseline=0,>=stealth,scale=.65]
\foreach \x in {1,...,6}{
\coordinate (\x) at (\x,0);
}
\foreach \x in {1,2,3}{
 \draw (\x) node[below,fill=white,draw=black,circle,inner sep =2] {} node[below,yshift=-5pt] {$\x$};
}
\foreach \x in {4,5,6}{
 \draw (\x) node[below,fill=black,draw=black,circle,inner sep =2] {} node[below,yshift=-5pt] {$\x$};
}
\draw[->] (1) to [in=90,out=90]  (4);
\draw[->] (1) to [in=90,out=90]  (5);
\draw[->] (2) to [in=90,out=90]  (4);
\draw[->] (3) to [in=90,out=90]  (5);
\draw[->] (3) to [in=90,out=90]  (6);
\end{tikzpicture}
&
\begin{tikzpicture}[baseline=0,>=stealth,scale=.65]
\foreach \x in {1,...,6}{
\coordinate (\x) at (\x,0);
}
\foreach \x in {1,2,3}{
 \draw (\x) node[below,fill=white,draw=black,circle,inner sep =2] {} node[below,yshift=-5pt] {$\x$};
}
\foreach \x in {4,5,6}{
 \draw (\x) node[below,fill=black,draw=black,circle,inner sep =2] {} node[below,yshift=-5pt] {$\x$};
}
\draw[->] (1) to [in=90,out=90]  (6);
\draw[->] (2) to [in=90,out=90]  (5);
\draw[->] (2) to [in=90,out=90]  (6);
\draw[->] (3) to [in=90,out=90]  (4);
\draw[->] (3) to [in=90,out=90]  (6);
\end{tikzpicture}
&
\begin{tikzpicture}[baseline=0,>=stealth,scale=.65]
\foreach \x in {1,...,6}{
\coordinate (\x) at (\x,0);
}
\foreach \x in {1,2,3}{
 \draw (\x) node[below,fill=white,draw=black,circle,inner sep =2] {} node[below,yshift=-5pt] {$\x$};
}
\foreach \x in {4,5,6}{
 \draw (\x) node[below,fill=black,draw=black,circle,inner sep =2] {} node[below,yshift=-5pt] {$\x$};
}
\draw[->] (1) to [in=90,out=90]  (4);
\draw[->] (1) to [in=90,out=90]  (5);
\draw[->] (2) to [in=90,out=90]  (4);
\draw[->] (2) to [in=90,out=90]  (6);
\draw[->] (3) to [in=90,out=90]  (5);
\end{tikzpicture}
\\
T_1 & T_2 & T_3
\end{array}
\]
\caption{Pairwise compatible alternating trees whose intersection $\C(T_1)\cap\C(T_2)\cap\C(T_3)$ is not full-dimensional.}\label{fig:badtriple}
\end{figure}
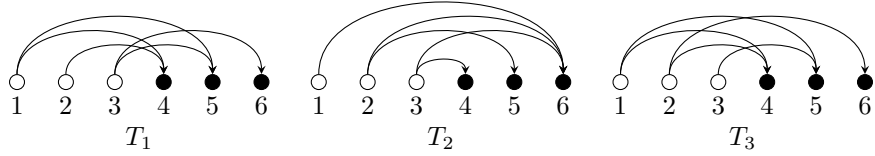

\begin{q}\label{q:cliques}
Which cliques in $\Gamma_n$ are/are not indexable collections?
\end{q}

While we have no good answer at the moment, we can say a bit more about maximal indexable collections in terms of the compatibility graph.
 
\begin{thm}\label{thm:maxclique}
All maximal indexable sets in $\In_n$ are maximal cliques in the compatibility graph $\Gamma_n$.
\end{thm}

By Corollary \ref{cor:indexableR}, we have the following comparison.

\begin{cor}\label{cor:indexableG}
The number of chambers $R_n$ of the rank $n$ resonance arrangement is bounded above by the number of maximal cliques of $\Gamma_{n+1}$.
\end{cor}

\begin{exam}
There are 250 alternating trees on $[5],$ so that $|V_5|=250.$  The graph $\Gamma_5$ is shown in Figure \ref{fig:W8}, where it is divided into 10 components of size one and 20 components of size 12.  There are 370 maximal cliques in this graph, each of which is a maximal set of indexable trees. Hence $R_4=370.$  

In $\Gamma_6$, however, there are $18,552$  maximal cliques yet only $11,296$ of these correspond to the maximal indexable sets labeling the chambers of $\Res_5$.
\end{exam}
  
\begin{figure}
\centering
 \includegraphics[width=.75\linewidth]{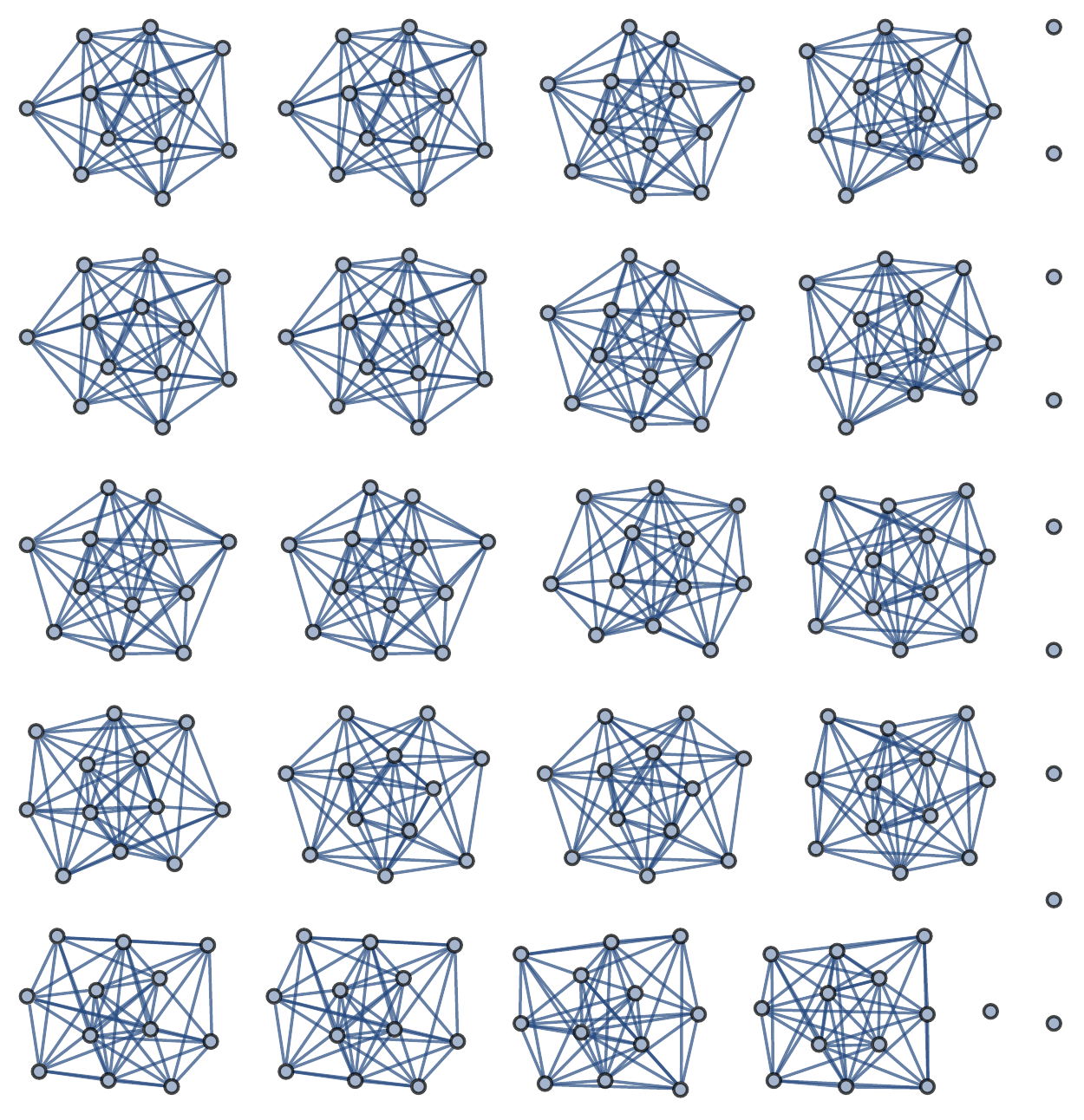}
\caption{The graph $\Gamma_5$ on 250 vertices (corresponding to the 250 alternating trees on 5 vertices). There are 20 isomorphic components on 12 vertices, and 10 singleton vertices.} \label{fig:W8}
\end{figure}

The proof of Theorem \ref{thm:maxclique} relies on the following lemma, which essentially says that any point in the interior of a root cone with $\sigma_I(\xx)=+$ can be induced by a positive flow on an alternating tree with $\sigma_I(T) = +$.

\begin{lem}\label{lem:?to+}
Suppose $T$ is an alternating tree on $[n]$, and suppose $\xx \in \C(T)$ is induced by a positive flow on $T$. Let $\emptyset \neq I \subsetneq [n]$ such that $\sigma_I(\xx)=+$. Then there exists an alternating tree $T'$ on $[n]$ such that $\sigma_I(T')=+$ and $\xx$ is also induced by a positive flow on $T'$.
\end{lem}

The proof of the lemma is a bit tedious, so we defer it to the next subsection. First, we present the proof of Theorem \ref{thm:maxclique}.

\begin{proof}[Proof of Theorem \ref{thm:maxclique}]
We will proceed by contradiction. Namely, suppose $\mathbf{T} = \{T_1, \ldots, T_k\}$ is a maximal clique in $\Gamma_n$ such that $\mathbf{T} \notin \In_n$, and yet for some $\ell < k$, $\mathbf{T}'=\{T_1,\ldots,T_{\ell}\} \in \In_n$. We consider $\ell$ maximial so that  $\{T_1,\ldots,T_{\ell}, T_{\ell+1}\} \notin \In_n$.

We will show there is a tree $T' \notin \mathbf{T}'$ such that $\C(T') \cap \C(T_1) \cap \cdots \cap \C(T_{\ell})$ is full-dimensional, contradicting the assertion that $\mathbf{T}'$ is a maximal indexable collection.

For each $j=1,\ldots,k$, we denote the intersection of the first $j$ root cones by
\[
 \C_j = \C(T_1) \cap \cdots \cap \C(T_j).
\] 
 
Since we are assuming that $\C_{\ell+1}$ is \emph{not} full-dimensional but $\C_{\ell}$ is,  there is a chamber $R$ of the resonance arrangement contained in $\C_{\ell}$, $R\subseteq \C_{\ell}$, but such that $R\not\subseteq \C_{\ell+1}$. Let $\xx$ be in the interior of $R$ so that the resonance sign vector $\sigma_S(\xx) \neq 0$ for all $\emptyset \subsetneq S \subseteq [n-1].$ 
 If,  for all $\emptyset \subsetneq S \subseteq [n-1],$ $\sigma_S(T_{\ell+1})\in \{\sigma_S(\xx), ?\},$ we would have that $R\subseteq \C(T_{\ell+1}).$  
Hence there must exist some $I\subseteq [n-1]$ such that $\sigma_I(T_{\ell+1})\neq \sigma_I(\xx)$ and $\sigma_I(T_{\ell+1}) \in \{+, -\}.$    Without loss of generality, assume that $\sigma_I(T_{\ell+1})=-$ so that $\sigma_I(\xx)=+.$  Then $\sigma_I(T_i)=?$ for $1\leq i\leq \ell$: For any such $i$, $R\subseteq \C_{\ell}$ implies $R\subseteq \C(T_i).$ Hence $\sigma_I(T_i) \in \{+, ?\}.$  But if $\sigma_I(T_i)=+,$ then $T_i$ and $T_{\ell+1}$ would not be sign compatible.

We will now construct a tree $T'$ with $T'\notin \mathbf{T}$ such that $\sigma_I(T') = +$ and $R\cap \C(T')$ is full-dimensional. The existence of this tree will complete our proof, since it will contradict the assertion that $\mathbf{T}'$ was a maximal indexable set.

Since $R \subseteq \C_{\ell}$, in particular $R\subseteq \C(T_1)$, and there is a point $\xx \in R$ induced by a positive flow on $T_1$. Since $\sigma_I(R) = +$, we know $\sigma_I(\xx)=+$. Since $\sigma_I(T_1) = ?$, there must be arcs of $T_1$ going both into and out of $I$, i.e., $\ind((T_1)_I)\neq \emptyset$. However, by Lemma \ref{lem:?to+}, we can modify $T_1$ to create a new tree $T'$ that also induces $\xx$ with a positive flow, such that $\ind(T'_I) = \emptyset$. This tree $T'$ satisfies our desired conditions: $R\subseteq \C(T')$ and $T' \notin \mathbf{T}$, thus completing the proof. 
\end{proof}

\subsection{Proof of Lemma \ref{lem:?to+}}

We will break the proof down into an even smaller technical lemma.

\begin{lem}\label{lem:swapflow}
Let $T$ be an alternating tree on $[n]$ and suppose $\xx$ is induced by a positive integer flow $f: E(T)\to \Z_{>0}$ on $T$ and that $\sigma_S(\xx)\neq 0$ for any $\emptyset \neq S \subset [n-1]$, i.e., $\xx$ is in the interior of a resonance chamber. Further, suppose $I$ is such that $\sigma_I(\xx)=+$ and let $(k,l)$ be an arc with $(k,l) \in \ind(T_I)$. Then there exists an alternating tree $T'$ on $[n]$ with $\ind(T'_I) \subseteq \ind(T_I)$ and with a positive integer flow $f': E(T') \to \Z_{>0}$ that also induces $\xx$, such that either:
\begin{itemize}
\item $(k,l) \notin E(T')$ or 
\item $f'(k,l) < f(k,l)$.
\end{itemize}
\end{lem}

\begin{proof}
As in the statement of the lemma, let $T$ be an alternating tree with positive integral flow $f$. Let $\xx=\xx(T;f)$ with $\sigma_S(\xx)\neq 0$ for any $S$. Further suppose $I$ is such that $\sigma_I(\xx)=+$ and let $(k,l) \in \ind(T_I)$.

Since $\sigma_I(\xx)=+$, the net flow out of $I$ is positive:
\[
 \sum_{(i,j)\in \outd(T_I)} f(i,j) - \sum_{(k,l) \in \ind(T_I)} f(k,l) = \sum_{i \in I} x_i > 0.
\]
In particular, since $f$ is a positive flow, $|\outd(T_I)| \neq 0$. Let $(i,j)\in \outd(T_I)$. Note that vertices $i, j, k,$ and $l$ must be distinct since $T$ is alternating: $i \in I$ and $k \notin I$ are sources, while $j \notin I$ and $l \in I$ are sinks.

We will create $T'$ in two stages. We first add an edge to $T$, creating a graph $T^+$ that contains a cycle. We will then augment the flows within the cycle, then delete an edge from $T^+$ to produce $T'$. The details of how we do this depends mildly on three cases. Let $P$ denote the unique undirected path from $k$ to $j$ in $T$, and let $T^+$ denote the alternating graph with arcs $E(T)\cup \{e\}$, where $e$ is the arc determined below.
\begin{enumerate}
\item if $P$ contains  $(k,l)$, then $e=(k,j)$,
\item if $P$ contains $(i,j)$ but not $(k,l)$, then $e=(i,l)$,
\item if $P$ contains neither $(i,j)$ nor $(k,l)$, then $e=(i,l)$.
\end{enumerate}
See Figure \ref{fig:lemmacases} for an illustration of each of these three cases for $P$. Notice the important feature that we can always form a cycle containing the edge $(k,l)$. Moreover, the graph $T^+$ is still alternating and the cycle is therefore even with at least four arcs.

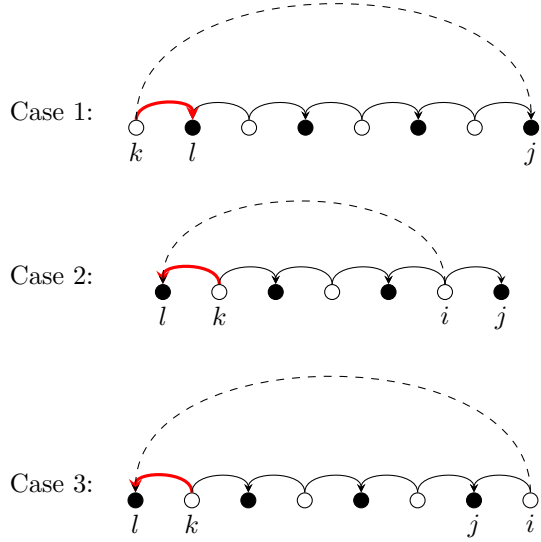
\begin{figure}
\[
\begin{array}{r c}
\mbox{Case 1:}
&
\begin{tikzpicture}[baseline=0,>=stealth,scale=.75]
\foreach \x in {1,...,8}{
\coordinate (\x) at (\x,0);
}
\foreach \x in {1,3,5,7}{
 \draw (\x) node[below,fill=white,draw=black,circle,inner sep =2] {};
}
\foreach \x in {2,4,6,8}{
 \draw (\x) node[below,fill=black,draw=black,circle,inner sep =2] {};
}
\draw[->] (3) to [in=90,out=90]  (2);
\draw[red,very thick,->] (1) to [in=90,out=90]  (2);
\draw[->] (3) to [in=90,out=90]  (4);
\draw[->] (5) to [in=90,out=90]  (4);
\draw[->] (5) to [in=90,out=90]  (6);
\draw[->] (7) to [in=90,out=90]  (6);
\draw[->] (7) to [in=90,out=90]  (8);
\draw[dashed,->] (1) to [in=90,out=90]  (8);
\draw (1) node[below, yshift=-5pt] {$k$};
\draw (2) node[below, yshift=-5pt] {$l$};
\draw (8) node[below, yshift=-5pt] {$j$};
\end{tikzpicture}
\\
\mbox{Case 2:}
&
\begin{tikzpicture}[baseline=0,>=stealth,scale=.75]
\foreach \x in {1,...,8}{
\coordinate (\x) at (\x,0);
}
\foreach \x in {2,4,6}{
 \draw (\x) node[below,fill=white,draw=black,circle,inner sep =2] {};
}
\foreach \x in {1,3,5,7}{
 \draw (\x) node[below,fill=black,draw=black,circle,inner sep =2] {};
}
\draw[red,very thick,->] (2) to [in=90,out=90]  (1);
\draw[->] (2) to [in=90,out=90]  (3);
\draw[->] (4) to [in=90,out=90]  (3);
\draw[->] (4) to [in=90,out=90]  (5);
\draw[->] (6) to [in=90,out=90]  (5);
\draw[->] (6) to [in=90,out=90]  (7);
\draw[dashed,->] (6) to [in=90,out=90]  (1);
\draw (2) node[below, yshift=-5pt] {$k$};
\draw (1) node[below, yshift=-5pt] {$l$};
\draw (6) node[below, yshift=-5pt] {$i$};
\draw (7) node[below, yshift=-5pt] {$j$};
\end{tikzpicture}
\\
\mbox{Case 3:}
&
\begin{tikzpicture}[baseline=0,>=stealth,scale=.75]
\foreach \x in {1,...,8}{
\coordinate (\x) at (\x,0);
}
\foreach \x in {2,4,6,8}{
 \draw (\x) node[below,fill=white,draw=black,circle,inner sep =2] {};
}
\foreach \x in {1,3,5,7}{
 \draw (\x) node[below,fill=black,draw=black,circle,inner sep =2] {};
}
\draw[red,very thick,->] (2) to [in=90,out=90]  (1);
\draw[->] (2) to [in=90,out=90]  (3);
\draw[->] (4) to [in=90,out=90]  (3);
\draw[->] (4) to [in=90,out=90]  (5);
\draw[->] (6) to [in=90,out=90]  (5);
\draw[->] (6) to [in=90,out=90]  (7);
\draw[->] (8) to [in=90,out=90]  (7);
\draw[dashed,->] (8) to [in=90,out=90]  (1);
\draw (2) node[below, yshift=-5pt] {$k$};
\draw (1) node[below, yshift=-5pt] {$l$};
\draw (8) node[below, yshift=-5pt] {$i$};
\draw (7) node[below, yshift=-5pt] {$j$};
\end{tikzpicture}
\end{array}
\]
\caption{The three types of paths $P$ from $k$ to $j$ in $T$, and how they are augmented in $T^+$ with a new arc. The new arc is drawn with a dashed line.}\label{fig:lemmacases}
\end{figure}

Let us denote the cycle in $T^+$ by
\[
 C = (e_1, e_2, \ldots, e_{2r})
\]
with $e_1 = (k,l)$. Note that the new arc $e$ we added is either $e_{2r}=(k,j)$ (Case 1) or it is $e_2 = (i,l)$ (Cases 2 and 3).

Let $f^*=\min\{f(e_1), f(e_3),\ldots, f(e_{2k-1})\}$. Note that $f^* \in \mathbb{Z}_{>0}$ since it was part of the original positive integral flow on $T$. 

The flow $f$ can be viewed as a flow on the arcs of $T^+$, $E(T^+) = E(T)\cup\{e\}$, with $f(e)=0$. We now create a new flow on $T^+$ by subtracting $f^*$ from all the odd-indexed arcs in $C$---this includes our special arc $e_1=(k,l)$---and adding $f^*$ to all the even-indexed arcs, i.e., let $f'$ be the flow given by
\[
 f'(a,b) = \begin{cases} 
    f(a,b) & \mbox{ if $(a,b) \notin C$},\\
    f(a,b)-f^* & \mbox{ if $(a,b) = e_{2s-1} \in C$ for some $1\leq s \leq r$},\\
    f(a,b)+f^* & \mbox{ if $(a,b) = e_{2s} \in C$ for some $1\leq s \leq r$}.
    \end{cases}
\]
This operation leaves the net flow at each vertex unchanged, and so induces the same point: $\xx(T^+;f') = \xx(T;f)$.

Moreover, we can see that one of the arcs of $C$ has to have flow zero, i.e., $f'(e_{2s-1})=0$ for some $s$. In fact, this arc is unique, for otherwise the nonzero parts of $f'$ would split into positive flows on two disjoint components, say $T^+_S$ and $T^+_{[n]-S}$. But then by Lemma \ref{lem:signflow}, $\sigma_S(\xx) = 0$, which contradicts our assumption that $\sigma_S(\xx) \neq 0$ for any $S$.

Now that we know the arc with flow zero is unique, we delete that arc, $e_{2s-1}$, from $T^+$ to obtain $T'$. Since $T^+$ was alternating, connected, and had exactly one cycle, and we removed one arc from that cycle, we know the graph $T'$ is indeed an alternating tree. Further, it satisfies all our desired properties: $\xx(T';f') = \xx$, and either $e_1=(k,l)$ was deleted, or $f'(k,l)=f(k,l)-f^* < f(k,l)$.
\end{proof}

We can now prove Lemma \ref{lem:?to+} by repeated application of Lemma \ref{lem:swapflow} and induction on $|\ind(T_I)|$. If $|\ind(T_I)| =0$, then $T=T'$ and we are done. Otherwise, suppose $\xx$ is induced by a positive integer flow on $T$. (By taking a nearby rational point and rescaling, it is safe make this assumption.) We then pick an arc $(k,l) \in \ind(T_I)$ and apply the lemma to produce a tree $T'$ with either $(k,l)\notin \ind(T'_I)$ or with $0<f'(k,l) < f(k,l)$. In the former case, we know $\ind(T'_I)\subsetneq \ind(T_I)$ so we are done by induction. 

In the latter case, we apply Lemma \ref{lem:swapflow} to $T'$ and the arc $(k,l)$ again, to produce a tree $T''$ with positive integer flow $f''$ such that (again) either $(k,l) \notin \ind(T''_I)$ or $0<f''(k,l)<f'(k,l)$. In at most $f(k,l)$ iterations, then, we will produce a tree with a positive integer flow and without arc $(k,l)$. 

By repeating the argument for any remaining arcs into $I$, Lemma \ref{lem:?to+} now follows.

\subsection{Symmetries of the compatibility graph and enumeration}

Since permutation of coordinates preserves adjacency of chambers, we can see some symmetries of sign vectors which carry over to the graph $\Gamma_{n+1}$. For example, consider the fact that if $T$ and $T'$ are sign compatible, then in particular they must have the same sources and the same sinks. Suppose they have $k$ sources among the nodes $\{1,2,\ldots,n\}$ (we ignore vertex $n+1$). Then by permuting labels/coordinates in $[n]$, there are sign compatible trees $\pi(T), \pi(T')$ with sources $\{1,2,\ldots,k\}$. Of course this reasoning applies to entire indexable sets, not just pairs of trees, and this narrows the focus of our counting problem.

Let us denote by $\In_{n+1}^I$ the set of maximal indexable collections whose trees $T$ have $\{ i \in [n] : i \mbox{ is a source vertex in } T \} = I$. In the sign vector for any such tree, we see $\sigma_J(T) = +$ for each $J\subseteq I$ and $\sigma_S(T)=-$ for $S \subseteq [n]-I$. Notice that for each $I$, the trees appearing in the collections for $\In_{n+1}^I$ either have $|I|$ or $|I|+1$ positive coordinates, depending on whether vertex $n+1$ is a source or sink. When $I$ is empty, there is precisely one alternating tree, with arcs from $n+1$ to every other vertex. Upon reversing arcs, we find there is but one tree with sources for all $I=[n]$. This symmetry of swapping sources and sinks (geometrically, multiplication by $-1$) extends to all other cases, so we find $\In_{n+1}^I$ and $\In_{n+1}^{[n]-I}$ are in bijection.

In the special case of $I=\{1,2,\ldots,k\}$, we write $\In_{n+1}^k = \In_{n+1}^{[k]}$ for short. The permutation of coordinates $x_1,\ldots,x_n$ mentioned earlier means the sets $\In_{n+1}^I$ and $\In_{n+1}^k$ are in bijection if $|I|=k$, with $\In_{n+1}^{\emptyset}=\In_{n+1}^0$. Thus we have
\[
 \In_{n+1} = \bigcup_{I \subseteq [n]} \In_{n+1}^I,
\]
and
\[
 R_n = |\In_{n+1}| = \sum_{k=0}^n \binom{n}{k} |\In_{n+1}^k|.
\]
The small values of $|\In_{n+1}^k|$ are given in Table \ref{tab:Ink}, where we witness the symmetry given by $\In_{n+1}^I \leftrightarrow \In_{n+1}^{[n]-I}$. For example, when $n=4$,
\[
 R_4 = 1 + 4\cdot 19 + 6\cdot 36 + 4\cdot 19 + 1 =370
\]

\begin{table}
\[
\begin{array}{c|ccccccc}
 n\backslash k & 0 & 1 & 2 & 3 & 4 & 5\\
 \hline
 0 & 1 \\
 1 & 1 & 1\\
 2 & 1 & 2 & 1 \\
 3 & 1 & 5 & 5 & 1 \\
 4 & 1 & 19 & 36 & 19 & 1 \\
 5 & 1 & 149 & 490 & 490 & 149 & 1 \\
 \end{array}
\]
\caption{The triangle of numbers $|\In_{n+1}^k|$.}\label{tab:Ink} 
\end{table}

This partitioning of the counting problem extends to the compatibility graph. Let us denote by $\Gamma_{n+1}^I$ the subgraph of $\Gamma_{n+1}$ consisting of all the alternating trees whose source set among the vertices $\{1,2,\ldots,n\}$ is $I$, and let $\Gamma_{n+1}^k = \Gamma_{n+1}^{[k]}$ for short, with $\Gamma_{n+1}^0 = \Gamma_{n+1}^{\emptyset}$. Then the graphs $\Gamma_{n+1}^I$ and $\Gamma_{n+1}^k$ are isomorphic, when $|I|=k$. For each $k=0,1,\ldots,n$, then, there are $\binom{n}{k}$ components of $\Gamma_{n+1}$ that are isomorphic to $\Gamma_{n+1}^k$.

We see that we can reconstruct all of $\Gamma_{n+1}$ from the disjoint components $\Gamma_{n+1}^k$, $k=1,\ldots,\lfloor n/2 \rfloor$. Let 
\[
h_k = |\{ \mbox{ maximal cliques in $\Gamma_n^k$ }\}|.
\]
Then $h_k=h_{n-k}$ and 
\[
 |\In_{n+1}^k| \leq h_k.
\]

\begin{exam}
Consider $\Gamma_5$ shown in Figure \ref{fig:W8} and how it relates to the chambers of $\Res_4$. The subgraph $\Gamma_5^0$ is an isolated node, isomorphic to $\Gamma_5^4$. It consists of the unique alternating tree in which vertex 5 is the only source. The graph $\Gamma_5^1$, isomorphic to $\Gamma_5^3$, has two connected components: an isolated node for the tree that has vertex 1 as its only source, and a connected component of 12 trees with source 1 and source 5. Finally, $\Gamma_5^2$ has two connected components: a component of 12 trees with sources $\{1,2\}$, and another component of 12 trees with sources $\{1,2,5\}$.

To build an isomorphic copy of the full graph $\Gamma_5$, we take:
\begin{itemize}
\item $2$ copies of $\Gamma_5^0$,
\item $2\cdot \binom{4}{1}$ copies of $\Gamma_5^1$, and
\item $2\cdot \binom{4}{2}$ copies of $\Gamma_5^2$.
\end{itemize}
Thus, we end up with $2+8=10$ isolated nodes and $8+12=20$ connected components with $12$ trees.
\end{exam}

Our symmetry so far focused on permutation of the coordinates $x_1,\ldots,x_n$ since these amount to symmetries of sign vectors. However, we can do a similar partition of $\Gamma_n$ by considering full permutations of $x_1,\ldots,x_{n+1}$ as well. To illustrate this idea, we return to $\Gamma_5$ and observe that there are $\binom{5}{1}+\binom{5}{4} = 10$ isolated nodes (corresponding to choosing either 1 or 4 nodes to be sources) and there are $\binom{5}{2}+\binom{5}{3}=20$ isomorphic components containing 12 trees each.

\section{Chambers of polynomiality for the Kostant partition function}\label{sec:Kostant}

We now turn our attention to the connection between chambers of the resonance arrangement and the chambers of polynomiality for the Kostant partition function. Let us first provide some background.

The \emph{Kostant partition function} (for the root system $A_n$) is a counting function $\kp_n : \R^{n+1}\to \Z_{\geq 0}$. For a given point $\v \in \R^{n+1}$, we have
\begin{equation} \label{kost} 
\kp_n(\v)= \left|\left\{ \xx \in \Z_{\geq 0}^{\binom{n+1}{2}} :  M(\kg^+_{n+1}) \xx=\v \right\}\right|,
\end{equation}
where $M(\kg_{n+1}^+)$ is the $(n+1)\times \binom{n+1}{2}$ incidence matrix of the complete graph $\kg^+_{n+1}$ with edges oriented from smaller to larger vertices. Thus the columns of $M(\kg_{n+1}^+)$ are precisely the positive roots $e_i-e_j$ with $1\leq i<j\leq n+1$, and the Kostant partition functions counts how many nonnegative integer flows on $\kg_{n+1}^+$ induce the point $\v$.
 
To put it another way, $\kp_n(\v)$ is the number of lattice points in the \emph{flow polytope}  $\F(\kg_{n+1}^+;\v)$ associated with the complete graph $\kg_{n+1}^+$ and netflow vector $\v$:
\[
\F(\kg_{n+1}^+;\v)= \left\{ \xx \in \R_{\geq 0}^{\binom{n+1}{2}} :  M(\kg_{n+1}^+) \xx= \v \right\}.
\] 
Kostant partition functions have a rich interplay with flow polytopes in algebraic combinatorics and combinatorial optimization as has been explored in, e.g., \cite{m-prod, BV2, bdlv, deloerasturm, Schrijver, tesler, mm,  AK}. The main connection with the resonance arrangement comes from the following result about the Kostant partition function.

\begin{thm}(\cite{deloerasturm}) \label{thm:deLS}
The Kostant partition function $\kp_n$ is a piecewise polynomial function of degree $\binom{n+1}{2}-n$. Its domains of polynomiality are the maximally refined chambers obtained by intersections of positive root cones.
\end{thm}

Compare this result with Proposition \ref{lem:chambersandroots} which states that resonance chambers are intersections of all (not necessarily positive) root cones. We can immediately infer a great deal about these chambers by restricting our study of alternating trees to the study of positive alternating trees, where we recall a \emph{positive} alternating tree has all arcs of the form $(i,j)$ with $i<j$. Recall also that the set of \emph{Kostant chambers} is denoted $\K_n$, and that the number of such chambers is $K_n = |\K_n|$. 

We adapt the notation and terminology of Section \ref{sec:resonance} as follows:
\begin{itemize}
\item A \textbf{positive indexable collection} $\textbf{T}=\{T_1,\ldots,T_k\}$ is an indexable collection of positive alternating trees.
\item We let $\In_n^+$ denote the set of maximal (under inclusion) positive indexable collections.
\item The \textbf{positive compatibility graph} $\Gamma_n^+$ is the subgraph of $\Gamma_n$ obtained by restricting the vertex set to positive alternating trees on $[n]$.
\end{itemize}
We have the following results from Section \ref{sec:resonance} mirrored for positive trees. 

\begin{cor}[Compare with Corollary \ref{cor:indexableR}]\label{cor:indexableR+}
The number of maximal indexable collections of positive trees on $[n+1]$ equals the number of chambers of polynomiality for the Kostant partition function, i.e.,
\[
 K_n = |\In^+_{n+1}|.
\]
\end{cor}

See Figure \ref{fig:labeledrank3+} to see the $K_3=7$ chambers of $\K_3$ labeled by maximal indexable collections of positive alternating trees on $[4]$. Compare with Figure \ref{fig:labeledrank3}.

\begin{cor}[Compare with Theorem \ref{thm:maxclique}]\label{cor:clique+}
All maximal positive indexable sets in $\In_n^+$ are cliques in the positive compatibility graph $\Gamma_n^+$.
\end{cor}

We caution that while $\Gamma_n^+ \subset \Gamma_n$, it is \emph{not} true that $\In_n^+$ is a subset of $\In_n$. Note that in particular it is not obvious whether the sets in $\In_n^+$ are \emph{maximal} cliques in $\Gamma_n^+$, and our proof of Theorem \ref{thm:maxclique} does not readily adapt itself to positive alternating trees. (In particular, if a new edge is created in Lemma \ref{lem:swapflow}, we cannot control whether it is of the form $(i,j)$ with $i<j$.) Thus as a follow-up to Question \ref{q:cliques} we propose the following question.

\begin{q}
Which cliques in $\Gamma_n^+$ are/are not positive indexable collections? In particular, is it true that all maximal positive indexable sets in $\In_n^+$ are maximal cliques in the positive compatibility graph $\Gamma_n^+$?
\end{q}

\begin{figure}
\[
\begin{tikzpicture}[scale=2.5]
\coordinate (a) at (0,4);
\coordinate (b) at (2,0);
\coordinate (c) at (-2,0);
\coordinate (d) at (1.55,2.36);
\coordinate (e) at (-1.55,2.36);
\coordinate (f) at (0,1.2);
\draw (a) to [in=90,out=-30]  (b);
\draw (b) to [in=-20,out=200]  (c);
\draw (c) to [in=210,out=90]  (a);
\draw (c) to [in=220,out=30] (d);
\draw (b) to [in=-40,out=150] (e);
\draw (a) to (f);
\draw (d) to [in=10,out=170] (e);
\draw (a) node[fill=black,circle, inner sep=2] {} node[above] {$e_2-e_3$};
\draw (b) node[fill=black,circle, inner sep=2] {} node[right] {$e_3-e_4$};
\draw (c) node[fill=black,circle, inner sep=2] {} node[left] {$e_1-e_2$};
\draw (d) node[fill=black,circle, inner sep=2] {} node[above right] {$e_2-e_4$};
\draw (e) node[fill=black,circle, inner sep=2] {} node[above left] {$e_1-e_3$};
\draw (f) node[fill=black,circle, inner sep=2] {} node[left,xshift=-3pt] {$e_1-e_4$};
\draw (0,.25) node {
\begin{tikzpicture}[baseline=0,>=stealth,scale=.5]
\foreach \x in {1,...,4}{
\coordinate (\x) at (\x,0);
}
\foreach \x in {1,3}{
 \draw (\x) node[below,fill=white,draw=black,circle,inner sep =2] {};
}
\foreach \x in {2,4}{
 \draw (\x) node[below,fill=black,draw=black,circle,inner sep =2] {};
}
\draw[->] (1) to [in=90,out=90] (2);
\draw[->] (1) to [in=90,out=90] (4);
\draw[->] (3) to [in=90,out=90] (4);
\end{tikzpicture}
};
\draw (1.2,1.2) node {
\begin{tikzpicture}[baseline=0,>=stealth,scale=.5]
\foreach \x in {1,...,4}{
\coordinate (\x) at (\x,0);
}
\foreach \x in {1,2,3}{
 \draw (\x) node[below,fill=white,draw=black,circle,inner sep =2] {};
}
\foreach \x in {4}{
 \draw (\x) node[below,fill=black,draw=black,circle,inner sep =2] {};
}
\draw[->] (1) to [in=90,out=90] (4);
\draw[->] (2) to [in=90,out=90] (4);
\draw[->] (3) to [in=90,out=90] (4);
\end{tikzpicture}
};
\draw (.45,2.15) node {
\begin{tabular}{c}
\begin{tikzpicture}[baseline=0,>=stealth,scale=.5]
\foreach \x in {1,...,4}{
\coordinate (\x) at (\x,0);
}
\foreach \x in {1,2}{
 \draw (\x) node[below,fill=white,draw=black,circle,inner sep =2] {};
}
\foreach \x in {3,4}{
 \draw (\x) node[below,fill=black,draw=black,circle,inner sep =2] {};
}
\draw[->] (1) to [in=90,out=90] (4);
\draw[->] (2) to [in=90,out=90] (3);
\draw[->] (2) to [in=90,out=90] (4);
\end{tikzpicture}
\\
\begin{tikzpicture}[baseline=0,>=stealth,scale=.5]
\foreach \x in {1,...,4}{
\coordinate (\x) at (\x,0);
}
\foreach \x in {1,2}{
 \draw (\x) node[below,fill=white,draw=black,circle,inner sep =2] {};
}
\foreach \x in {3,4}{
 \draw (\x) node[below,fill=black,draw=black,circle,inner sep =2] {};
}
\draw[->] (1) to [in=90,out=90] (3);
\draw[->] (1) to [in=90,out=90] (4);
\draw[->] (2) to [in=90,out=90] (4);
\end{tikzpicture}
\end{tabular}
};
\draw (.5,3) node {
\begin{tabular}{c}
\begin{tikzpicture}[baseline=0,>=stealth,scale=.5]
\foreach \x in {1,...,4}{
\coordinate (\x) at (\x,0);
}
\foreach \x in {1,2}{
 \draw (\x) node[below,fill=white,draw=black,circle,inner sep =2] {};
}
\foreach \x in {3,4}{
 \draw (\x) node[below,fill=black,draw=black,circle,inner sep =2] {};
}
\draw[->] (1) to [in=90,out=90] (4);
\draw[->] (2) to [in=90,out=90] (3);
\draw[->] (2) to [in=90,out=90] (4);
\end{tikzpicture}
\\
\begin{tikzpicture}[baseline=0,>=stealth,scale=.5]
\foreach \x in {1,...,4}{
\coordinate (\x) at (\x,0);
}
\foreach \x in {1,2}{
 \draw (\x) node[below,fill=white,draw=black,circle,inner sep =2] {};
}
\foreach \x in {3,4}{
 \draw (\x) node[below,fill=black,draw=black,circle,inner sep =2] {};
}
\draw[->] (1) to [in=90,out=90] (3);
\draw[->] (2) to [in=90,out=90] (3);
\draw[->] (2) to [in=90,out=90] (4);
\end{tikzpicture}
\end{tabular}
};
\draw (-.5,3) node {
\begin{tabular}{c}
\begin{tikzpicture}[baseline=0,>=stealth,scale=.5]
\foreach \x in {1,...,4}{
\coordinate (\x) at (\x,0);
}
\foreach \x in {1,2}{
 \draw (\x) node[below,fill=white,draw=black,circle,inner sep =2] {};
}
\foreach \x in {3,4}{
 \draw (\x) node[below,fill=black,draw=black,circle,inner sep =2] {};
}
\draw[->] (1) to [in=90,out=90] (4);
\draw[->] (1) to [in=90,out=90] (3);
\draw[->] (2) to [in=90,out=90] (3);
\end{tikzpicture}
\\
\begin{tikzpicture}[baseline=0,>=stealth,scale=.5]
\foreach \x in {1,...,4}{
\coordinate (\x) at (\x,0);
}
\foreach \x in {1,2}{
 \draw (\x) node[below,fill=white,draw=black,circle,inner sep =2] {};
}
\foreach \x in {3,4}{
 \draw (\x) node[below,fill=black,draw=black,circle,inner sep =2] {};
}
\draw[->] (1) to [in=90,out=90] (3);
\draw[->] (2) to [in=90,out=90] (3);
\draw[->] (2) to [in=90,out=90] (4);
\end{tikzpicture}
\end{tabular}
};
\draw (-.45,2.15) node {
\begin{tabular}{c}
\begin{tikzpicture}[baseline=0,>=stealth,scale=.5]
\foreach \x in {1,...,4}{
\coordinate (\x) at (\x,0);
}
\foreach \x in {1,2}{
 \draw (\x) node[below,fill=white,draw=black,circle,inner sep =2] {};
}
\foreach \x in {3,4}{
 \draw (\x) node[below,fill=black,draw=black,circle,inner sep =2] {};
}
\draw[->] (1) to [in=90,out=90] (4);
\draw[->] (1) to [in=90,out=90] (3);
\draw[->] (2) to [in=90,out=90] (3);
\end{tikzpicture}
\\
\begin{tikzpicture}[baseline=0,>=stealth,scale=.5]
\foreach \x in {1,...,4}{
\coordinate (\x) at (\x,0);
}
\foreach \x in {1,2}{
 \draw (\x) node[below,fill=white,draw=black,circle,inner sep =2] {};
}
\foreach \x in {3,4}{
 \draw (\x) node[below,fill=black,draw=black,circle,inner sep =2] {};
}
\draw[->] (1) to [in=90,out=90] (3);
\draw[->] (1) to [in=90,out=90] (4);
\draw[->] (2) to [in=90,out=90] (4);
\end{tikzpicture}
\end{tabular}
};
\draw (-1.2,1.2) node {
\begin{tikzpicture}[baseline=0,>=stealth,scale=.5]
\foreach \x in {1,...,4}{
\coordinate (\x) at (\x,0);
}
\foreach \x in {1}{
 \draw (\x) node[below,fill=white,draw=black,circle,inner sep =2] {};
}
\foreach \x in {2,3,4}{
 \draw (\x) node[below,fill=black,draw=black,circle,inner sep =2] {};
}
\draw[->] (1) to [in=90,out=90] (2);
\draw[->] (1) to [in=90,out=90] (3);
\draw[->] (1) to [in=90,out=90] (4);
\end{tikzpicture}
};
\end{tikzpicture}
\]
\caption{The 7 Kostant chambers of $\K_3$, labeled by indexable collections of positive alternating trees on $[4]$. Compare with Figure \ref{fig:labeledrank3}.}\label{fig:labeledrank3+}
\end{figure}
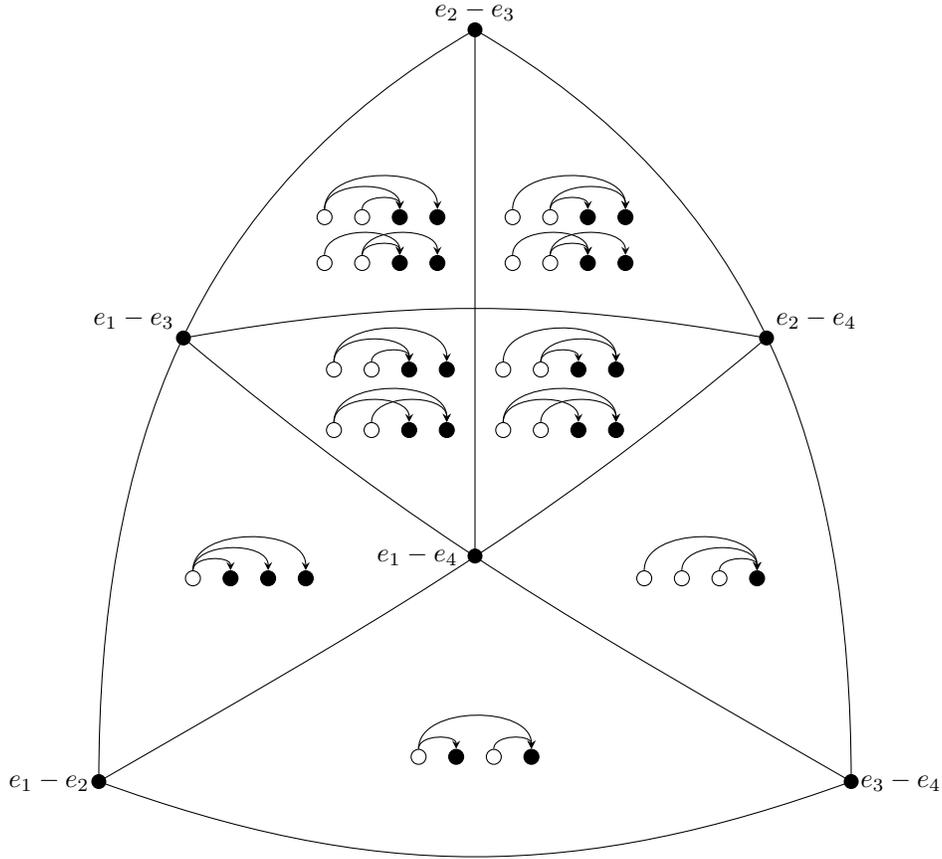
 
From Proposition \ref{lem:chambersandroots} and Theorem \ref{thm:deLS} the Kostant chambers contain resonance chambers, and by the Corollary \ref{cor:cyclicsymmetry} about the consequences of cyclic permutation of coordinates, we obtain the following upper bound recorded also in Observation \ref{obs:Kn}:

\begin{cor}\label{thm:Kn}
The resonance chambers in the positive root cone $\C(\kg_n^+)$ refine the Kostant chambers. In particular,
\[
 K_n \leq R_n^+ = \frac{R_n}{(n+1)}.
\]
\end{cor}

As stated in Problem \ref{conj:bounds} in the introduction, we have also observed
\[
 R_n < K_{n+1} < \frac{1}{2}T_n,
\]
on small data points, though we cannot prove this.  

\section{The threshold arrangement}\label{sec:threshold}

 The threshold arrangement $\Th_n$ (corresponding to $T_{n-1}$) has normal vectors given by all $\pm 1$ vectors in $\R^{n}$ (corners of an $n$-cube). Since a nonzero vector $v$ and its opposite $-v$ give the same hyperplane, we can choose as representative normal vectors those vectors 
\[
v_S = (\pm 1, \pm 1, \ldots, \pm 1, -1),
\]
where the elements of $S \subseteq [n-1]=\{1,2,\ldots,n-1\}$ indicate which entries are positive.  For example, if $n=8$,
\[
 v_{\{1,3,4,6\}} = (1,-1,1,1,-1,1,-1,-1). 
\]
Let $V_S = \{\xx \in \R^{n} : \langle \xx, v_S \rangle = 0\}$ denote the hyperplane normal to $v_S$. Let $\Th_{n}$ denote the arrangement of these $2^{n-1}$ hyperplanes, and let $T_{n-1}$ denote the number of chambers in this arrangement.  We call this the \emph{threshold arrangement} since its chambers are in bijection with threshold functions on $n-1$ variables (see, e.g., \cite{Zuev92}). 

For example, in Figure \ref{fig:thresharr3b} (also in Figure \ref{fig:thresharr3}) we see the threshold arrangement of rank $3$. Here, the four hyperplanes are
\begin{align*}
V_{\emptyset} &= \{(x,y,z)\in \R^3 : x+y+z=0\},\\
V_1 &= \{(x,y,z)\in \R^3 : x=y+z\},\\ 
V_2 &= \{(x,y,z)\in \R^3 : y=x+z\},\\
V_{12} &= \{(x,y,z)\in \R^3 : x+y=z\}.
\end{align*}
The normal vectors are $v_{\emptyset} = (-1,-1,-1)$, $v_1 = (1,-1,-1)$, $v_2 = (-1,1,-1)$, and $v_{12} = (1,1,-1)$.

\begin{figure}
\[
\begin{tikzpicture}[scale=1,baseline=1cm]
\draw[draw=none,fill=white!80!black,opacity=1] (0,0)--(3,3)--(11,3)--(8,0);
\draw[thick] (6,0)--(5,3);
\draw[thick] (1.5,1.5)--(9.5,1.5);
\draw[thick] (1,0)--(10,3);
\draw[draw=none, fill=white!80!black,opacity=.5] (.25,2.25)--(9.25,5.25)--(10.5,1.5)--(1.5,-1.5)--(.25,2.25);
\draw[draw=none, fill=white!80!black,opacity=.5] (4.5,-2)--(7.5,2)--(6.5,5)--(3.5,1)--(4.5,-2);
\draw (5.5,1.5)--(6.7,4.4);
\draw[draw=none, fill=white!80!black,opacity=.5] (.5,3.5)--(2.5,-.5)--(10.5,-.5)--(8.5,3.5)--(.5,3.5);
\draw (5.5,1.5) node[fill=black, circle, inner sep=1] {};
\draw (5.5,1.5)--(7,3.5);
\draw (5.5,1.5)--(4,3.5);
\draw (.5,3.5)--(8.5,3.5)--(10.5,-.5);
\draw (1.5,-1.5)--(.25,2.25)--(9.25,5.25);
\draw (6.5,5)--(7.5,2)--(4.5,-2);
\draw (1,0)--(5.5,1.5);
\draw (7.15,1.5)--(9.5,1.5);
\draw (6,0)--(5.5,1.5);
\draw (10,3)--(8.92,2.64);
\draw (0,0)--(8,0)--(11,3);
\draw (8,0) node[above left, fill=none, inner sep=1] {$V_{\emptyset}$};
\draw (10,-1) node[above left, fill=none, inner sep=1] {$V_{1}$};
\draw (10,4) node[above, fill=none, inner sep=1] {$V_{2}$};
\draw (5,-2) node[fill=none, inner sep=1] {$V_{12}$};
\end{tikzpicture}
\]
\caption{A view of the threshold arrangement $\Th_3$ of rank 3. The six regions of the resonance arrangement $\Res_2$ can be seen as the restrictions of $V_1, V_2$, and $V_{12}$ to the subspace $V_{\emptyset}$.}\label{fig:thresharr3b}
\end{figure}
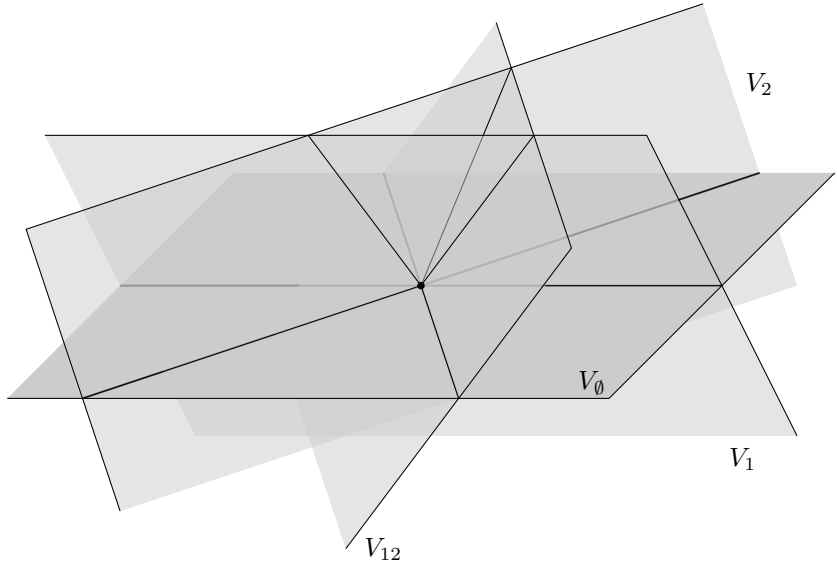

\begin{defn}[Threshold sign vectors]\label{def:threshvector}
Given a point $\xx \in \R^n$, the \textbf{threshold sign vector} of $\xx$ is denoted by
\[
 \tau(x) = (\tau_S(x))_{S \subseteq [n-1]},
\]
where 
\[
\tau_S(\xx) = \begin{cases} + & \mbox{if } \langle \xx, v_S \rangle > 0,\\
 - & \mbox{if } \langle \xx, v_S \rangle < 0,\\
 0 & \mbox{if } \langle \xx, v_S \rangle = 0.
 \end{cases}
\]
\end{defn}

For example, the point $\xx=(1,2,1)$ has $\tau(\xx)$ given by
\[
 (\tau_{\emptyset},\tau_1, \tau_2, \tau_{12}) = (-,-,0,+).
\]

\subsection{Invariance of the threshold arrangement}

The arrangement $\Th_n$ is invariant under flipping signs in coordinates, i.e., under reflections across coordinate hyperplanes (in fact, it is invariant under the action of the hyperoctahedral group of signed permutations of coordinates; this fact was deployed by Zuev \cite{Zuev92} in the proof of a lower bound of $T_n$). This implies that the face structure in any particular hyperplane of $\Th_n$ is the same as the face structure of $\Th_n$ in $V_{\emptyset}$. We make this claim precise now.

First we note the following lemma about sign vectors.

\begin{lem}\label{lem:ts}
For any $I, J \subseteq [n-1]$, and $\xx \in V_J$, we have
\begin{equation}\label{eq:signs}
 \frac{\langle \xx, v_I \rangle}{2} = \langle \xx, u_I-u_J \rangle = \langle \xx, u_I \rangle -\langle \xx, u_J \rangle.
\end{equation}
In particular, if $\xx \in \Res_{n-1}$, then $\sigma_I(\xx)=\tau_I(\xx)$.
\end{lem}

The lemma says that if $J = \emptyset$, i.e., if $\xx$ is in $\Res_{n-1}$, then the sign vectors coincide: $\tau_I(\xx) = \sigma_I(\xx)$. The only difference is that in $\tau(\xx)$ we also note that $\tau_{\emptyset}(\xx) = 0$.

\begin{proof}
Let $\xx \in V_J$. Then by definition,
\[
 \langle \xx, v_J \rangle = \sum_{i \in J} x_i - \sum_{j \in [n]-J} x_j = 0,
\]
and so
\[
 x_n = \sum_{i \in J} x_i - \sum_{j \in [n-1]-J} x_j.
\]
Thus,
\begin{align*}
 \langle \xx, v_I \rangle &= \sum_{i \in I} x_i - \sum_{j \in [n-1]-I} x_j - x_n, \\
  &= \sum_{i \in I} x_i - \sum_{j \in [n-1]-I} x_j + \sum_{j \in [n-1]-J} x_j - \sum_{i \in J} x_i, \\
  &= 2\sum_{i \in I-J} x_i - 2\sum_{j \in J-I} x_j,\\
  &= 2\sum_{i \in I-J} x_i + 2\sum_{k \in I \cap J} x_k - 2\sum_{k \in I \cap J} x_k- 2\sum_{j \in J-I} x_j,\\
  &= 2\sum_{i \in I} x_i - 2\sum_{j \in J} x_j,\\
  &= 2 \langle \xx, u_I \rangle -2\langle \xx, u_J \rangle.
\end{align*}
\end{proof}

Now, let $r_i(x)$ denote the reflection that sends $\xx$ to $\xx-2\langle \xx, e_i\rangle e_i$, where $e_i$ is the standard basis vector, i.e., we subtract $2x_i$ from the $i$th coordinate of $\xx$:
\[
 r_i((x_1,\ldots,x_i, \ldots, x_n)) = (x_1,\ldots,-x_i,\ldots,x_n).
\]
This ``toggles'' the sign of the $i$th entry and leaves the rest of the vector untouched.

Notice that $r_i^2(\xx) = \xx$, so the toggle is an involution (also clear since it's just the reflection across the coordinate hyperplane), and if $i\neq j$,
\[
 r_i(r_j(\xx)) =(x_1,\ldots, -x_i,\ldots,-x_j,\ldots,x_n) =r_j(r_i(\xx)),
\]
so these toggles commute. Since the toggles commute, it makes sense to write 
\[
r_J(\xx) = \xx-\sum_{i \in J} 2\langle \xx, e_i\rangle e_i,
\]
for any subset $J \subseteq [n-1]$. Again, $r_J^2(\xx) = \xx$.

In what follows, let $I\Delta J = (I\cup J)-(I \cap J)$ denotes the symmetric difference of two sets.

\begin{lem}\label{lem:ST}
 For any $\xx \in \R^n$, and any subsets $I, J \subseteq [n-1]$, we have
 \[
  \tau_I(\xx) = \tau_{I\Delta J}(r_J(\xx)).
 \] 
\end{lem}

\begin{proof} Let $\mathbf{y} = r_J(\xx)$. Notice that $y_i = x_i$ if $i \notin J$, and $y_j=-x_j$ if $j \in J$. Then,
\begin{align*}
 \langle \mathbf{y}, v_{I \Delta J} \rangle &= \sum_{i \in I \Delta J } y_i - \sum_{j \in [n]-I\Delta J} y_j,\\
 &= \sum_{i \in I - J } y_i + \sum_{k \in J-I} y_k - \sum_{j \in [n]-(I\cup J)} y_j - \sum_{l \in I\cap J} y_l,\\
 &= \sum_{i \in I - J } x_i - \sum_{k \in J-I} x_k - \sum_{j \in [n]-(I\cup J)} x_j + \sum_{l \in I\cap J} x_l,\\
 &= \sum_{i \in I} x_i - \sum_{j \in [n]-I} x_j,\\
 &= \langle \xx, v_I \rangle.
\end{align*}
\end{proof}

If we take $I=J$ in Lemma \ref{lem:ST}, we see that if $\xx \in V_J$, then $r_J(\xx) \in V_{\emptyset}$. This, along with Lemma \ref{lem:ts}, implies the following observation.

\begin{obs}\label{obs:proj}
If $\xx \in V_J$, then $r_J(\xx) \in V_{\emptyset}$. Moreover, $\sigma_I(r_J(\xx)) = \tau_I(r_J(\xx)) = \tau_{I\Delta J}(\xx)$ for any subset $I \subseteq [n-1]$.
\end{obs}

In other words, for $\xx \in V_J$, the $\tau$-sign vector of $\xx$ determines $\sigma$-sign vector of $r_J(\xx)$.

For example, if $\xx = (-1,3,1,1) \in V_{12}$, then it has $\tau(\xx)$ given by
\[
 (\tau_{\emptyset}, \tau_1, \tau_2, \tau_3, \tau_{12}, \tau_{13}, \tau_{23}, \tau_{123}) = (-, -, +, -, 0, -, +, +).
\]
We have $r_{12}(\xx) = (1,-3,1,1) \in V_{\emptyset}$, and $\sigma(\xx)$ is given by 
\[
 (\sigma_{\emptyset}, \sigma_1, \sigma_2, \sigma_3, \sigma_{12}, \sigma_{13}, \sigma_{23}, \sigma_{123}) = (0,+,-,+,-,+,-,-).
\]
Note $\sigma_{\emptyset} = \tau_{12}$, $\sigma_1 = \tau_{2}$, $\sigma_2 = \tau_1$, $\sigma_3=\tau_{123}$, and so on.

We now collect the important consequences of our lemmas and observations.

\begin{cor}\label{cor:wallbij}
For any face $F$ that is not a chamber of $\Th_n$, we have $F \in V_I$ if and only if $r_J(F) \in V_{I\Delta J}$. Moreover, the sign vector $\tau(F)$ uniquely determines $\tau(r_J(F))$, and vice versa.

In particular, for each face $F \in \Res_{n-1}$, we have $r_J(F) \in V_J$ and the sign vector $\tau(r_J(F))$ is uniquely determined by the sign vector $\tau(F)=\sigma(F)$. 

Thus, $r_J$ gives a bijection between the walls of $\Th_n$ in hyperplane $V_J$ and the chambers of the resonance arrangement $\Res_{n-1}$.
\end{cor}

\begin{proof}
 The first claim is an immediate consequence of Lemma \ref{lem:ST}, since $(I\Delta J)\Delta J = I$. The second claim uses the first claim together with Lemma \ref{lem:ts} that shows the $\tau$- and $\sigma$-vectors coincide in $V_{\emptyset}$. The third statement refers to only codimension one faces of $\Th_n$.
\end{proof}

We are now ready to prove the bounds in Theorem \ref{thm:Rn}.

\subsection{Threshold functions and resonance chambers}

In this section we prove Theorem \ref{thm:Rn}. For convenience we restate the inequality \eqref{eq:ineq} claimed in Theorem \ref{thm:Rn}:
\[
 \frac{(n+1)}{2^{n+1}}T_n < R_n < \frac{1}{2}T_n.
\]

The upper bound follows from Observation \ref{obs:hyp} part (1). That is, since $\Res_{n}$ lives in a hyperplane of $\Th_{n+1}$, its chambers are walls in $\Th_{n+1}$ that separate chambers with $\tau_{\emptyset} > 0$ and $\tau_{\emptyset} < 0$. Thus, for each of the $R_n$ chambers of $\Res_n$ there are two chambers in $\Th_{n+1}$ that contain it on their boundary. This immediately implies
 \[
   2R_n < T_n.
 \]
The bound is not sharp, as can be seen in the $n=3$ case, where there are two chambers with no walls on $V_{\emptyset}$.
 
The lower bound in Theorem \ref{thm:Rn} follows this idea: 
 \[
 (\mbox{chambers in } \Th_n) \hookrightarrow (\mbox{walls in } \Th_n) \leftrightarrow (\mbox{chambers in } \Res_{n-1}) \times (\mbox{hyperplanes in } \Th_n).
 \]
 From Observation \ref{obs:hyp} part (2), we know $nC(\Th_n) \leq 2W(\Th_n)$, where $C(\Th_n) = T_{n-1}$ is the number of chambers of the threshold arrangement, and $W(\Th_n)$ is the number of walls in the arrangement. The inequality is strict since the threshold arrangement is not simplicial (and so there are chambers with more than $n$ walls on their boundary). The walls are partitioned according to the hyperplane they live in, so that $W(\Th_n) = \sum_{I\subseteq [n-1]} W(V_I)$. But, by Corollary \ref{cor:wallbij}, we know $W(V_I) = C(\Res_{n-1}) = R_{n-1}$ for all subsets $I$. Hence,
 \[
  nT_{n-1} < 2W(\Th_n) = 2\sum_{I\subseteq [n-1]} W(V_I) = 2^{n} R_{n-1},
 \]
from which the lower bound claimed in Theorem \ref{thm:Rn} follows:
\[
 \frac{nT_{n-1}}{2^n} < R_{n-1}.
\]

\begin{remark}[Intertwined arrangements]
In this paper we focus on how the resonance arrangement sits inside the threshold arrangement. Curiously, we also note that the threshold arrangement of lower rank embeds in the resonance arrangement as well, implying $T_{n-1} < R_n$. This bound also yields the asymptotic result for $\log_2 R_n$ in Theorem \ref{thm:Rn}, but it is known that $(2^{n-1}+1)T_{n-1} \leq T_n$ (see \cite{Yajima65}), and $T_n/(2^{n-1}+1) < (n+1)T_n/2^{n+1}$ for $n$ larger than $3$, so the lower bound in \eqref{eq:ineq} is better than $T_{n-1}$. 
\end{remark}

\begin{remark}[Better upper bounds]
One can do better than the upper bound in Theorem \ref{thm:Rn} if one understands how many walls to expect in a typical chamber of $\Th_n$. That is, $w(n)T_{n-1} = 2W(\Th_n) = 2^n R_{n-1}$, where $w(n)$ is the average number of walls per chamber. Since $\Th_n$ has rank $n$, $w(n)\geq n$. While neither the threshold arrangement nor the resonance arrangement are simplicial, $w(n)$ might not grow too quickly with $n$. For example, if $w(n)<n^2$ this would say 
\[
 2^n R_{n-1} < n^2 T_{n-1},
\]
which would imply the upper bound:
\[
 R_{n-1} < \frac{n^2T_{n-1}}{2^n}.
\]
The data for (the base-2 logarithm of) this comparison is given in Tables \ref{tab:wndata}, which seems to suggest that $w(n)$ is closer to $n$ than $n^2$.
\end{remark}

\begin{p} \label{p2} 
Estimate $w(n)$, the average number of walls per chamber in $\Th_n$. In particular, is $w(n)<n^2$? 
\end{p}

\begin{table}
\begin{center}
\begin{tabular}{|c||c|c|c||c|} \hline
$n$ & $\log_2 \left( \lfloor \frac{(n+1)}{2^{n+1}}T_n\rfloor \right)$ & $\log_2 \left( R_n\right)$ & $\log_2 \left( \frac{1}{2}T_n\right)$ & $\log_2 \left(\lceil \frac{(n+1)^2}{2^{n+1}}T_n \rceil \right)$\\ 
 \hline \hline
 1 & 1 & 1 & 1 & 2\\
 2 & 2.4 & 2.6 & 2.8 & 4.0\\
 3 & 4.7 & 5 & 5.7 & 6.7\\
 4 & 8.2 & 8.5 & 9.9 & 10.5\\
 5 & 13.1 & 13.5 & 15.5 & 15.7\\
 6 & 19.6 & 20.0 & 22.8 & \textbf{22.5}\\
 7 & 28.0 & 28.4 & 32.0 & \textbf{31.0} \\
 8 & 38.2 & 38.6 & 43.0 & \textbf{41.3}\\
 \hline
\end{tabular}
\end{center}
\caption{Base-2 logarithms of the number of maximal unbalanced families and lower and upper bounds in terms of the number of threshold function. Boldface entries are better than the best general upper bound.}\label{tab:wndata}
\end{table}

\section*{Acknowledgments}

The authors are grateful to the American Institute of Mathematics for hosting the workshop on ``Polyhedral geometry and partition theory" in November 2016 where a conversation of the last two authors sparked the idea for this project. The aforementioned conversation stemmed from an inspiring presentation of Jesus de Loera about the number of chambers of polynomiality of the Kostant partition function.  The authors  are grateful to Jesus de Loera for the mentioned talk as well as for further helpful communications and to Lou Billera for numerous motivating exchanges about this project. The first two authors thank Luca Moci for several discussions related to this project.

\bibliographystyle{plain}
\bibliography{bibliog2}

\end{document}